\documentclass[11pt]{article}
\usepackage{amssymb}
\usepackage{amsbsy}
\usepackage[latin1]{inputenc}
\usepackage{amsthm}
\usepackage[dvips]{graphicx}
\usepackage{graphicx} 
\usepackage{subfigure}
\usepackage{pst-eucl}
\usepackage[latin1]{inputenc}
\usepackage[english]{babel}
\usepackage{amsmath,amssymb,graphics,mathrsfs}
\usepackage{amsmath,amssymb,latexsym}
\usepackage{graphicx,color}
\usepackage[T1]{fontenc}
\usepackage[active]{srcltx}
\usepackage{multicol}
\usepackage[latin1]{inputenc}
\usepackage{pst-all}
\usepackage{enumerate}
\usepackage{pstricks}
\usepackage{pstricks-add}
\usepackage{setspace}
\usepackage{soul}
\usepackage{cancel}
\usepackage{nonfloat}
\usepackage[margin=10pt,font=footnotesize,labelfont=bf,labelsep=endash]{caption}
\usepackage[left=4cm,top=3cm,right=2.4cm,bottom=3.2cm]{geometry}
\usepackage[colorlinks=true,citecolor=red,linkcolor=blue,urlcolor=RubineRed,pdfpagetransition=Blinds,pdftoolbar=false,pdfmenubar=false]{hyperref}

\newcommand{\R}{\hbox{\rm I \kern -5pt R}}     
\newcommand{\p} {\hbox{\rm I \kern -5pt P}}

\def\x        {\textit{\textbf{x}}}
\def\W        {\textit{\textbf{W}}}

\def\H        {{\boldsymbol H}}

\newtheorem{prop}{Proposition}[section]
\newtheorem{defi}[prop]{Definition}
\newtheorem{tma}[prop]{Theorem}
\newtheorem{cor}[prop]{Corollary}
\newtheorem{obs}[prop]{Remark}
\newtheorem{lem}[prop]{Lemma}

\allowdisplaybreaks

\begin{document}

\title{Unconditionally energy stable fully discrete schemes for a chemo-repulsion model}
\author{F.~Guillén-González\thanks{Dpto. Ecuaciones Diferenciales y An\'alisis Num\'erico and IMUS, 
Universidad de Sevilla, Facultad de Matemáticas, C/ Tarfia, S/N, 41012 Sevilla (SPAIN). Email: guillen@us.es, angeles@us.es},
M.~A.~Rodríguez-Bellido$^*$~and
 D.~A.~Rueda-Gómez$^*$\thanks{Escuela de Matemáticas, Universidad Industrial de Santander, A.A. 678, Bucaramanga (COLOMBIA). Email:  diaruego@uis.edu.co}}

\date{}
\maketitle

\begin{abstract}

This work is devoted to study unconditionally energy stable and mass-conservative numerical schemes for the following repulsive-productive chemotaxis model: Find $u \geq 0$, the cell density, and $v \geq 0$, the chemical concentration, such that
$$
\left\{
\begin{array}
[c]{lll}%
\partial_t u - \Delta u - \nabla\cdot (u\nabla v)=0 \ \ \mbox{in}\ \Omega,\ t>0,\\
\partial_t v - \Delta v + v =  u \ \ \mbox{in}\ \Omega,\ t>0,
\end{array}
\right.
$$
in a bounded domain $\Omega\subseteq \mathbb{R}^d$, $d=2,3$. By using a regula\-ri\-za\-tion technique, we propose three fully discrete Finite Element (FE) approximations. The first one is a non\-li\-near approximation in the variables $(u,v)$; the second one is another nonlinear approximation obtained by introducing ${\boldsymbol\sigma}=\nabla v$ as an auxiliary variable; and the third one is a linear approximation constructed by mixing the regularization procedure with the energy quadratization technique, in which other auxiliary variables are introduced. In addition, we study the well-posedness of the nume\-ri\-cal schemes, proving unconditional existence of solution, but conditional uniqueness (for the nonlinear schemes). Finally, we compare the behavior of  such schemes throughout several numerical simulations and provide some conclusions.
\end{abstract}



\section{Introduction}
Chemotaxis is a biological phenomenon in which the movement of living organisms is induced by a chemical stimulus. The chemotaxis is called attractive when the organisms move towards regions with higher chemical concentration, while if the motion is towards lower concentrations, the chemotaxis is called repulsive. In this paper, we study unconditionally energy stable fully discrete schemes for the following parabolic-parabolic repulsive-productive chemotaxis model (with linear production term):
\begin{equation}  \label{modelf00}
\left\{
\begin{array}
[c]{lll}%
\partial_t u - \Delta u = \nabla\cdot (u\nabla v)\ \ \mbox{in}\ \Omega,\ t>0,\\
\partial_t v - \Delta v + v =  u
 \ \mbox{in}\ \ \Omega,\ t>0,\\
\displaystyle\frac{\partial u}{\partial \mathbf{n}}=\frac{\partial v}{\partial \mathbf{n}}=0\ \ \mbox{on}\ \partial\Omega,\ t>0,\\
u(\x,0)=u_0(\x)\geq 0,\ v(\x,0)=v_0(
\x)\geq 0\ \ \mbox{in}\ \Omega,
\end{array}
\right. \end{equation}
in a bounded domain $\Omega\subseteq \mathbb{R}^d$, $d=2,3$, with boundary $\partial \Omega$. The unknowns for this model are $u(\x, t) \geq 0$, the cell density, and $v(\x, t) \geq 0$, the chemical concentration. Problem (\ref{modelf00}) is conservative in $u$, because the total mass $\int_\Omega u(\cdot,t)$ remains constant in time, as we can check  integrating equation (\ref{modelf00})$_1$ in $\Omega$, 
\begin{equation}\label{conservuInt}
\frac{d}{dt}\left(\int_\Omega u(\cdot,t)\right)=0, \ \ \mbox{ i.e. } \ 
\int_\Omega u(\cdot,t)=\int_\Omega u_0 := m_0, \ \ \forall t>0.
\end{equation}
Problem (\ref{modelf00}) is well-posed \cite{Cristian}: In 3D domains, there exist global in time nonnegative weak solutions of model (\ref{modelf00}) in the following sense:
\begin{equation*}
\begin{array}{ccc}
u\in C_w([0,T];L^1(\Omega)) \cap L^{5/4}(0,T;W^{1,5/4}(\Omega)),\ \ \forall T>0,\\
v\in L^\infty(0,T; H^1(\Omega)) \cap L^{2}(0,T;H^{2}(\Omega))\cap C([0,T]; L^2(\Omega)), \ \ \forall T>0,\\
\partial_t u\in L^{4/3}(0,T; W^{1,\infty}(\Omega)'),\ \ \partial_t v\in L^{5/3}(0,T; L^{5/3}(\Omega)), \ \ \forall T>0,
\end{array}
\end{equation*}
 satisfying the following variational formulation of the $u$-equation
\begin{equation*}
\int_0^T \langle \partial_t u,\bar{u}\rangle + \int_0^T (\nabla u,  \nabla \bar{u}) +\int_0^T (u\nabla v,\nabla \bar{u})=0, \ \ \forall \bar{u}\in L^4(0,T;W^{1,\infty}(\Omega)), \ \ \forall T>0,
\end{equation*}
and the $v$-equation pointwisely
\begin{equation*}
\partial_t v -\Delta v + v=u \ \ \mbox{ a.e. } (t,\x)\in (0,+\infty)\times\Omega.
\end{equation*}
Moreover, for 2D domains, there exists a unique classical and bounded in time solution. A key step of the existence proof in \cite{Cristian} is to establish an energy equality, which in a formal manner, is obtained as follows: if we consider 
\begin{equation*}
F(s):= s(ln s - 1)+1\geq 0 \ \Rightarrow F'(s)=ln\, s \ \Rightarrow F''(s)=s^{-1}, \ \ \forall s>0,
\end{equation*}
then multiplying (\ref{modelf00})$_1$ by $F'(u)$, (\ref{modelf00})$_2$ by $-\Delta v$, integrating over $\Omega$, using (\ref{modelf00})$_3$ and adding, the chemotactic and production terms cancel, and we obtain
\begin{eqnarray}\label{deluvintro}
\frac{d}{dt} \displaystyle \int_\Omega \Big( F(u) + \frac{1}{2} \vert \nabla {v}\vert^2\Big) d \x + \int_\Omega \Big(4\vert \nabla (\sqrt{u})\vert^2 +\vert \Delta v\vert^2 + \vert \nabla v\vert^2\Big) d \x = 0.
\end{eqnarray}
The aim of this work is to design numerical methods for model (\ref{modelf00}) conserving, at the discrete level, the mass-conservation and energy-stability properties of the continuous model (see (\ref{conservuInt})-(\ref{deluvintro}), respectively). There are only a few works about numerical analysis for chemotaxis models. For instance, for the Keller-Segel system (i.e.~with chemo-attraction and linear production),  Filbet studied in \cite{Filbet} the existence of discrete solutions and the convergence of a finite volume scheme. Saito, in \cite{Saito1,Saito2}, proved error estimates for a conservative Finite Element (FE) approximation.  A mixed FE approximation  is studied in \cite{Marrocco}. In \cite{Eps}, some error estimates are proved for a fully discrete discontinuous FE method. In the case where the chemotaxis occurs in heterogeneous medium, in \cite{CST} the convergence of a combined finite volume-nonconforming finite element scheme is studied, and some discrete properties are proved.\\

Some previous energy stable numerical schemes have also been studied in the chemotaxis framework. A finite volume scheme for a Keller-Segel model with an additional cross-diffusion term satisfying the energy-stablity property (that means, a discrete energy decreases in time) has been studied in \cite{BJ}. Unconditionally energy stable time-discrete numerical schemes and fully discrete FE schemes for a chemo-repulsion model with quadratic production has been analyzed in \cite{FMD,FMD2} respectively. However, as far as we know, for the chemo-repulsion model with linear production (\ref{modelf00}) there are not works studying energy-stable schemes. We emphasize that the numerical analysis of energy stability in the chemo-repulsion model with linear production has greater difficulties than the case of quadratic production \cite{FMD,FMD2}. In fact, in the continuous case of quadratic production, in order to obtain an energy equality, it is necessary to test the $u$-equation by $u$, and the $v$-equation by $-\Delta v$, which, if we want to move to the fully discrete approximation, is much easier than the case of linear production in which, as it was said before, the energy equality is obtained multiplying the $u$-equation by the nonlinear function $F'(u)=ln\, u$. \\

In this paper, we propose three unconditional energy stable fully discrete schemes, in which, in order to obtain rigorously a discrete version of the energy law (\ref{deluvintro}), we argue through a re\-gu\-larization technique. This regularization procedure has been used in previous works to deal with the test function $F'(u)=ln\, u$ in fully discrete approximations, as for example, for a cross-diffusion competitive population model \cite{BB}  or a cross-diffusion segregation problem arising from a model of interacting particles \cite{GS}. The model that will be analyzed in this paper differs primarily from these previous works in the fact that, in our case, the term of self-diffusion in (\ref{modelf00})$_1$ is $\nabla \cdot(\nabla u)$ and it is not in the form $\nabla \cdot(u\nabla u)$ as in \cite{BB,GS}, which makes the analysis a bit more difficult. In fact, in the continuous problem, if we multiply equation (\ref{modelf00})$_1$ by $F'(u)=ln\, u$, in our case we obtain the dissipative term $\int_\Omega \frac{1}{u}\vert \nabla u\vert^2$ (which does not provide an estimate for $\nabla u$), while in the cases of \cite{BB,GS}, it is obtained $\int_\Omega \vert \nabla u\vert^2$ which gives directly an estimate for $\nabla u$ in $L^2(\Omega)$.\\

The outline of this paper is as follows: In Section 2, we give the notation and define the regularized functions that will be used in the fully discrete approximations. In Section 3, we study a nonlinear fully discrete FE approximation of (\ref{modelf00}) in the original variables $(u,v)$. We prove the well-posedness of the numerical approximation, and show the mass-conservation and energy-stability properties of this scheme by imposing the orthogonality condition on the mesh (see ({\bf H}) below). In Section 4, we analyze another nonlinear FE approximation obtained by introducing ${\boldsymbol\sigma=\nabla v}$ as an auxiliary variable, and again, we prove the well-posedness of the scheme, as well as its mass-conservation and energy-stability properties, but without imposing the orthogonality condition ({\bf H}). In Section 5, we study a linear fully discrete FE approximation constructed by mixing the regularization procedure with the Energy Quadratization (EQ) strategy, in which the energy of the system is transformed into a quadratic form by introducing new auxiliary variables. This EQ technique has been applied to different fields such as liquid crystals \cite{BGJ,ZYGW}, phase fields \cite{ZYGZ} (and references therein) and molecular beam epitaxial growth \cite{YZW} models, among others. Finally, in Section 6, we compare the behavior of the schemes throughout several numerical simulations, and provide some conclusions in Section 7.

\section{Notation and preliminary results}\label{NRe}
First, we recall some functional spaces which will be used throughout this paper. We will consider the usual Sobolev spaces $H^m(\Omega)$ and Lebesgue spaces $L^p(\Omega),$
$1\leq p\leq \infty,$ with norms $\Vert\cdot\Vert_{m}$ and $\Vert\cdot \Vert_{L^p}$, respectively. In particular,  the $L^2(\Omega)$-norm will be denoted by $\Vert
\cdot\Vert_0$. Throughout $(\cdot,\cdot)$ denotes the standard $L^2$-inner product over $\Omega$. We denote by $\H^{1}_{\sigma}(\Omega):=\{{\boldsymbol\sigma}\in \H^{1}(\Omega): {\boldsymbol\sigma}\cdot \mathbf{n}=0 \mbox{ on } \partial\Omega\}$ and we will use the following equivalent norms in $H^1(\Omega)$  and ${\bf H}_{\sigma}^1(\Omega)$, respectively (see \cite{necas} and \cite[Corollary 3.5]{Nour}, respectively):
\begin{equation*}
\Vert u \Vert_{1}^2=\Vert \nabla u\Vert_{0}^2 + \left( \int_\Omega u\right)^2, \ \ \forall u\in H^1(\Omega),
\end{equation*}
\begin{equation*}
\Vert {\boldsymbol\sigma} \Vert_{1}^2=\Vert {\boldsymbol\sigma}\Vert_{0}^2 + \Vert \mbox{rot }{\boldsymbol\sigma}\Vert_0^2 + \Vert \nabla \cdot {\boldsymbol\sigma}\Vert_0^2, \ \ \forall {\boldsymbol\sigma}\in \H^{1}_{\sigma}(\Omega),
\end{equation*}
where rot ${\boldsymbol\sigma}$ denotes the well-known rotational operator (also called curl) which is scalar for 2D domains and vectorial for 3D ones. If $Z$ is a
general Banach space, its topological dual space will be denoted by $Z'$.
Moreover, the letters $C,K$ will denote different positive
constants which may change from line to line (or even within the same
line). \\
In order to construct energy-stable fully discrete schemes for problem (\ref{modelf00}), we are going to follow a regularization procedure. We will use the approach introduced by Barrett and Blowey \cite{BB}. Let $\varepsilon\in (0,1)$ and consider the truncated function $\lambda_\varepsilon:\mathbb{R}\rightarrow [\varepsilon,\varepsilon^{-1}]$ given by
\begin{equation}\label{aE}
\lambda_\varepsilon(s)\ := \ 
\left\{\begin{array}{lcl}
\varepsilon & \mbox{ if } & s\leq \varepsilon,\\
s & \mbox{ if } & \varepsilon\leq s\leq \varepsilon^{-1},\\
\varepsilon^{-1} & \mbox{ if } & s\geq \varepsilon^{-1}.
\end{array}\right. 
\end{equation}
If we define 
\begin{equation} \label{F2pE}
F''_\varepsilon(s):= \frac{1}{\lambda_\varepsilon(s)},
\end{equation}
then, we can integrate twice in (\ref{F2pE}), imposing the conditions $F'_\varepsilon(1)=F_\varepsilon(1)=0$, and we obtain a convex function $F_\varepsilon: \mathbb{R}\rightarrow [0,+\infty)$, such that $F_\varepsilon \in  C^{2,1}(\mathbb{R})$ (see Fig. \ref{fig:Fe}). Even more, for $\varepsilon\in (0,e^{-2})$, it holds \cite{BB}
\begin{equation}\label{PNa}
F_\varepsilon(s)\geq \frac{\varepsilon}{2}s^2 - 2\ \ \forall s\geq 0 \ \  \mbox{  and } \ \ F_\varepsilon(s)\geq \frac{s^2}{2\varepsilon} \ \ \forall s\leq 0.
\end{equation}
\begin{figure}[htbp] 
\centering 
\subfigure[$\lambda_\varepsilon(s)$ vs  $\frac{1}{F''(s)}:=s$]{\includegraphics[width=72mm]{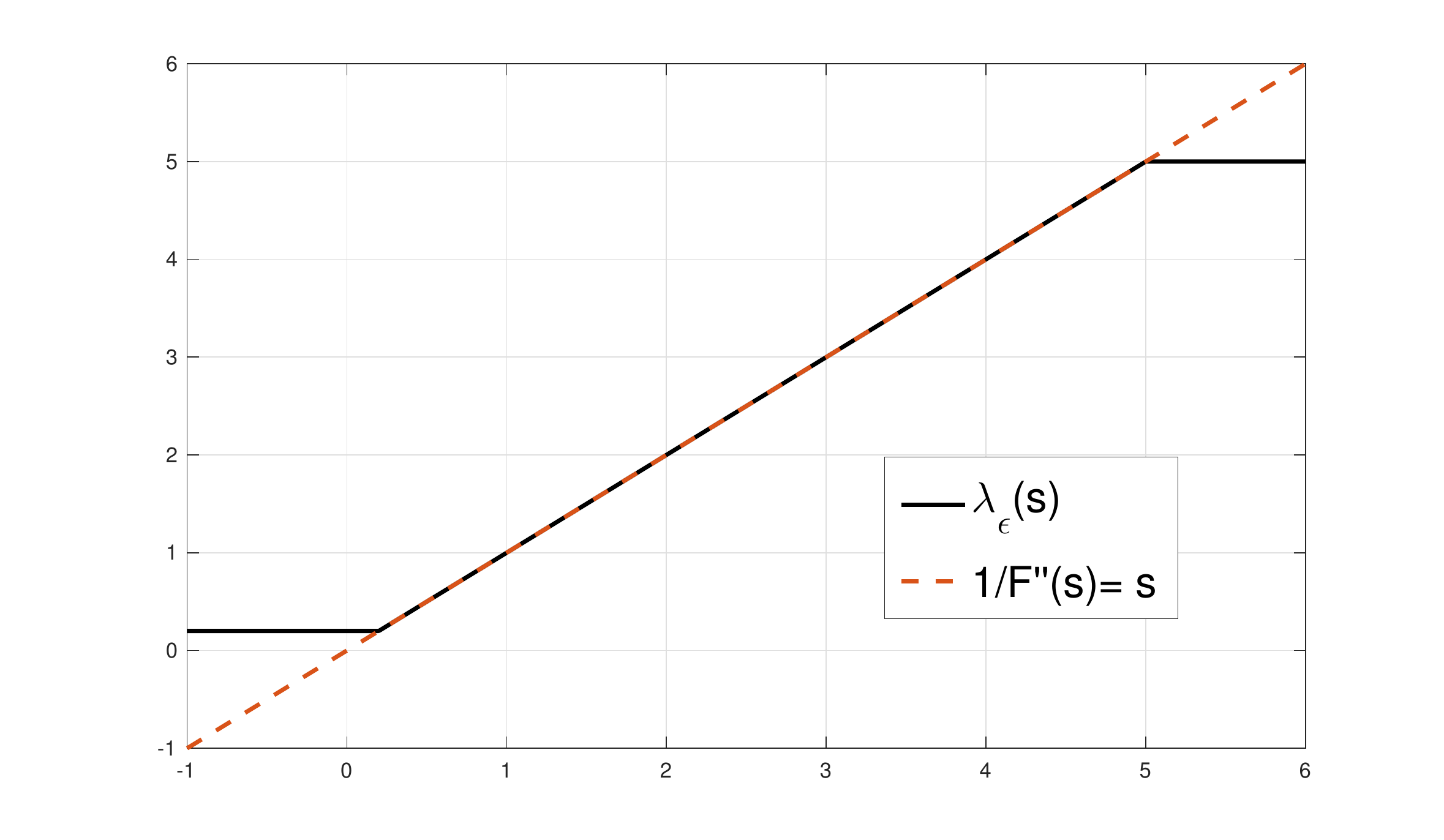}} 
\subfigure[$F''_\varepsilon(s)$ vs  $F''(s):= \frac{1}{s}$]{\includegraphics[width=72mm]{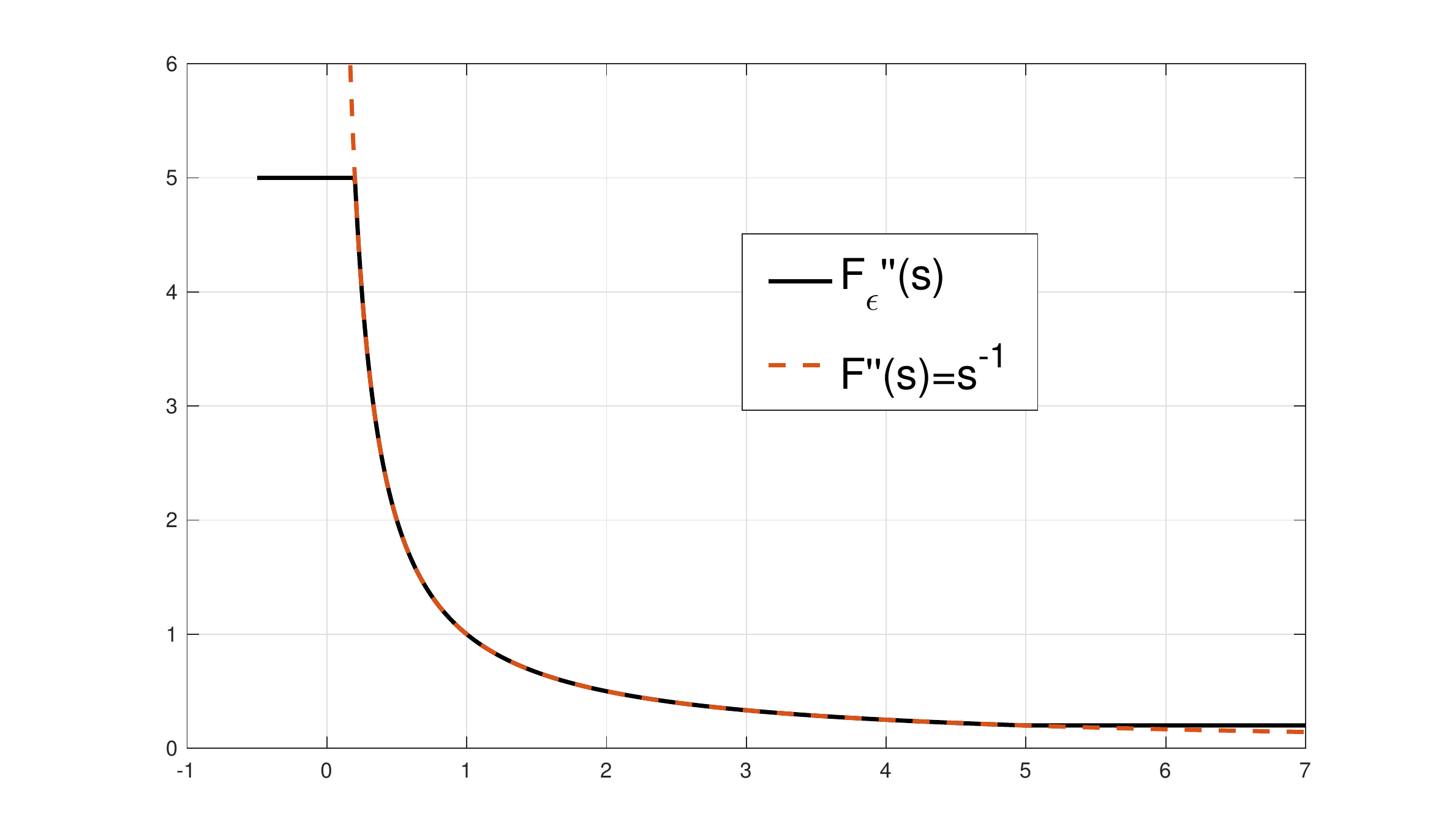}}
\subfigure[$F'_\varepsilon(s)$ vs  $F'(s):= ln\; s$]{\includegraphics[width=72mm]{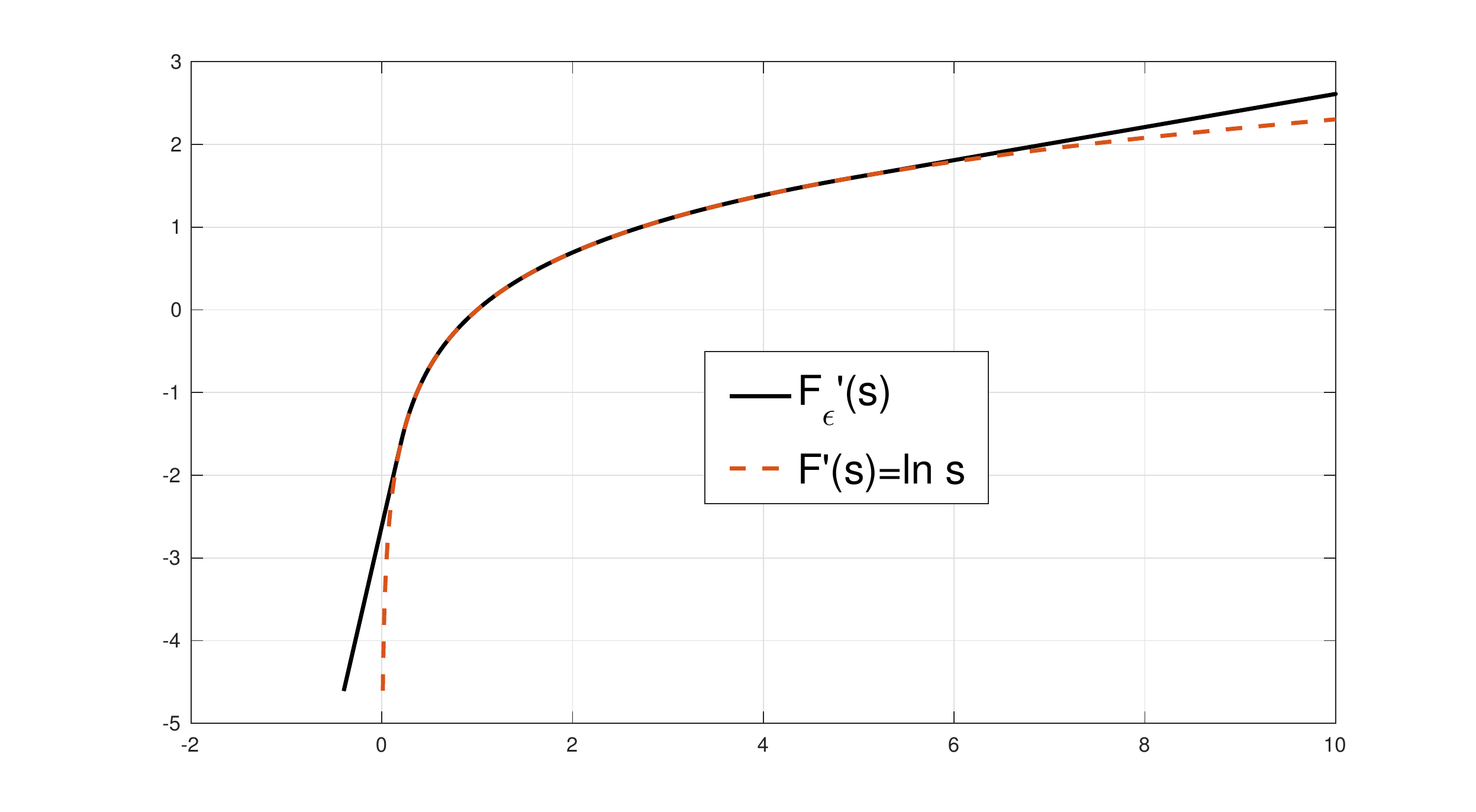}} 
\subfigure[$F_\varepsilon(s)$ vs  $F(s):= s(ln\; s -1)+1$]{\includegraphics[width=72mm]{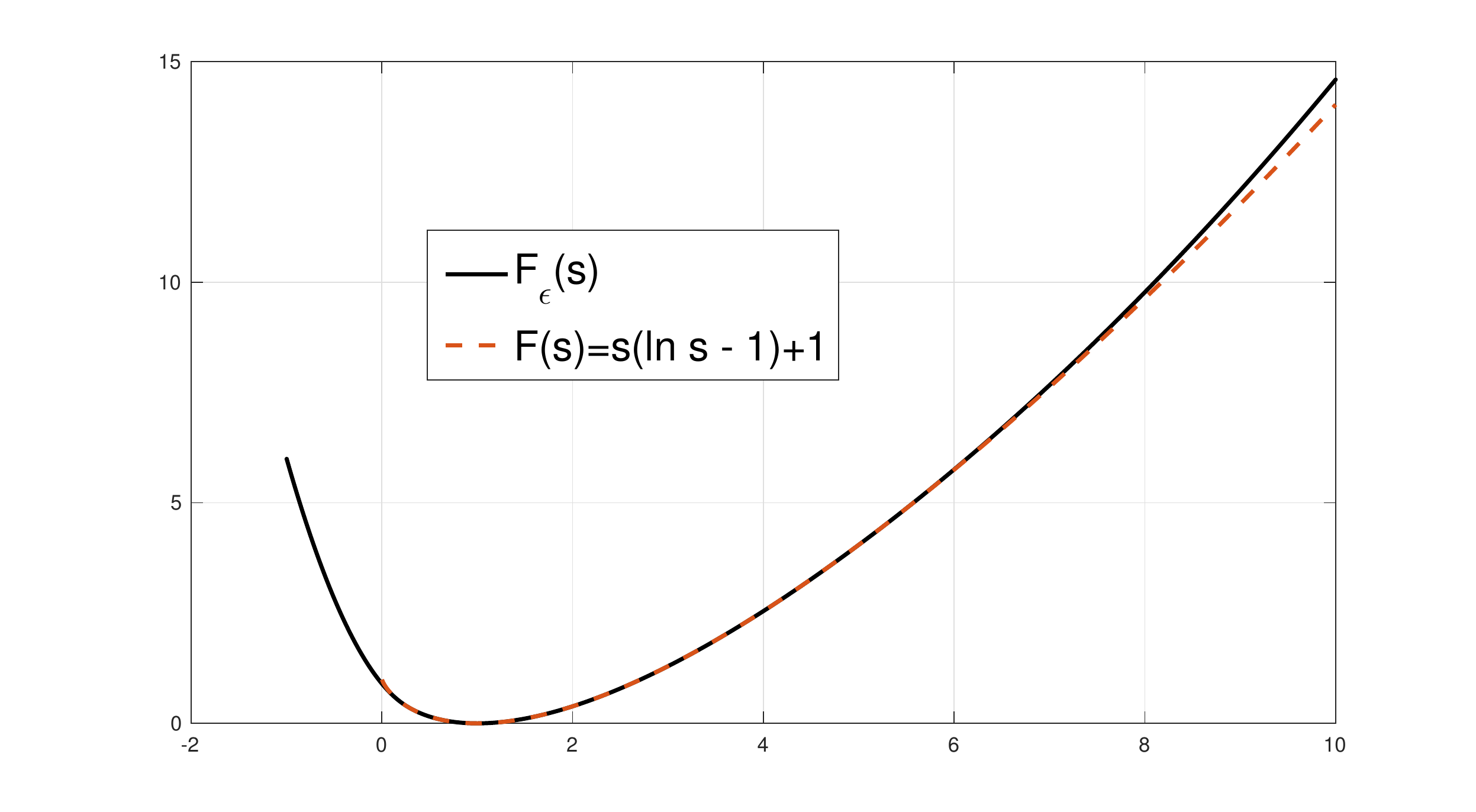}}
\caption{Functions $\lambda_\varepsilon$ and $F_\varepsilon$ and its derivatives.} \label{fig:Fe}
\end{figure}
Finally, we will use the following result to get large time estimates \cite{He}:
\begin{lem} \label{tmaD}
Assume that $\delta,\beta,k>0$ and $d^n\geq 0$ satisfy
\begin{equation*}
 (1+\delta k)d^{n+1} \leq d^n + k\beta, \ \ \forall n\geq 0.
\end{equation*}
Then, for any $n_0\geq 0$,
\begin{equation*}
d^n \leq (1+\delta k)^{-(n-n_0)} d^{n_0} + \delta^{-1} \beta, \ \  \forall n\geq n_0.
\end{equation*}
\end{lem}

\section{Scheme UV}
In this section, we propose an energy-stable nonlinear fully discrete scheme (in the variables $(u,v)$) associated to model (\ref{modelf00}). With this aim, taking into account the functions $\lambda_\varepsilon$ and $F_\varepsilon$ and its derivatives, we consider the following regularized version of problem (\ref{modelf00}): Find $u_\varepsilon, v_\varepsilon:\Omega\times [0,T]\rightarrow \mathbb{R}$ such that
\begin{equation}
\left\{
\begin{array}
[c]{lll}%
\partial_t u_\varepsilon -\Delta u_\varepsilon -\nabla\cdot(\lambda_\varepsilon (u_\varepsilon)\nabla {v}_\varepsilon)=0\ \ \mbox{in}\ \Omega,\ t>0,\\
\partial_t {v}_\varepsilon -\Delta v_\varepsilon + {v}_\varepsilon = u_\varepsilon\ \ \mbox{in}\ \Omega,\ t>0,\\
\displaystyle \frac{\partial u_\varepsilon}{\partial \mathbf{n}}=\frac{\partial v_\varepsilon}{\partial \mathbf{n}}=0\ \ \mbox{on}\ \partial\Omega,\ t>0,\\
u_\varepsilon(\x,0)=u_0(\x)\geq 0,\ v_\varepsilon(\x,0)=v_0(
\x)\geq 0\ \ \mbox{in}\ \Omega.
\end{array}
\right.  \label{modelf02acont}
\end{equation}
\begin{obs}
The idea is to define a fully discrete scheme associated to (\ref{modelf02acont}), taking in general $\varepsilon=\varepsilon(k,h)$, such that $\varepsilon(k,h)\rightarrow 0$ as $(k,h)\rightarrow 0$.
\end{obs}
Observe that multiplying (\ref{modelf02acont})$_1$ by $F'_\varepsilon(u_\varepsilon)$, (\ref{modelf02acont})$_2$ by $-\Delta v_\varepsilon$, integrating over $\Omega$ and adding, again the chemotactic and production terms cancel, and we obtain the following energy law
\begin{eqnarray*}
\frac{d}{dt} \displaystyle \int_\Omega \Big( F_\varepsilon(u_\varepsilon) + \frac{1}{2} \vert \nabla {v}_\varepsilon\vert^2\Big) d \x + \int_\Omega \Big(F''_\varepsilon(u_\varepsilon)\vert \nabla u_\varepsilon\vert^2 +\vert \Delta v_\varepsilon\vert^2 + \vert \nabla v_\varepsilon\vert^2\Big) d \x = 0.
\end{eqnarray*}
In particular, the modified energy $\mathcal{E}_\varepsilon(u,v)= \displaystyle\int_\Omega \Big( F_\varepsilon(u) + \frac{1}{2} \vert \nabla {v}\vert^2\Big) d \x$ is decreasing in time. Then, we consider a fully discrete approximation using FE in space and backward Euler in time (considered for simplicity on a uniform partition of $[0,T]$ with time step
$k=T/N : (t_n = nk)_{n=0}^{n=N}$). Let $\Omega$ be a polygonal domain. We consider a shape-regular and quasi-uniform family of triangulations of $\Omega$, denoted by $\{\mathcal{T}_h\}_{h>0}$, with simplices $K$, $h_K= diam(K)$ and $h:= \max_{K\in \mathcal{T}_h} h_K$, so that $\overline{\Omega}=\cup_{K\in \mathcal{T}_h} \overline{K}$. Moreover, in this case we will assume the following hypothesis:
\begin{enumerate}
\item[({\bf H})]{The triangulation is structured in the sense that all simplices have a right angle.}
\end{enumerate}
We choose the following continuous FE spaces for $u$ and $v$:
$$(U_h, V_h) \subset H^1(\Omega)^2 ,\quad \hbox{generated by $\mathbb{P}_1,\mathbb{P}_m$ with $m\geq 1$.}
$$
\begin{obs}
The right angled requirement and the choice of $\mathbb{P}_1$-continuous FE for $U_h$ are nece\-ssa\-ry in order to obtain the relation (\ref{PL1}) below, which is essential in order to prove the energy-stability of the scheme \textbf{UV} (see Theorem \ref{estinc1uv} below).
\end{obs}
Let $J$ be the set of vertices of $\mathcal{T}_h$ and $\{p_j\}_{j\in J}$ the coordinates of these vertices. We denote the Lagrange interpolation operator by $\Pi^h: C(\overline{\Omega})\rightarrow U_h$, and we introduce the discrete semi-inner product on $C(\overline{\Omega})$ (which is an inner product in $U_h$) and its induced discrete seminorm (norm in $U_h$):
\begin{equation}\label{mlump}
(u_1,u_2)^h:=\int_\Omega \Pi^h (u_1 u_2), \   \vert u \vert_h=\sqrt{(u,u)^h}.
\end{equation}
\begin{obs}\label{eqh2}
In $U_h$, the norms $\vert \cdot\vert_h$ and $\Vert \cdot\Vert_0$ are equivalents uniformly with respect to $h$ (see \cite{PB}).
\end{obs}
We consider also the  $L^2$-projection $Q^h:L^2(\Omega)\rightarrow U_h$ given by
\begin{equation}\label{MLP2}
(Q^h u,\bar{u})^h=(u,\bar{u}), \ \ \forall \bar{u}\in U_h,
\end{equation}
and the standard $H^1$-projection $R^h:H^1(\Omega)\rightarrow V_h$. Moreover, for each $\varepsilon\in (0,1)$ we consider the construction of the operator $\Lambda_\varepsilon: U_h\rightarrow L^\infty(\Omega)^{d\times d}$ given in  \cite{BB}, satisfying that $\Lambda_\varepsilon u^h$ is a symmetric and positive definite matrix  for all $u^h\in U_h$ and a.e. $\x$ in $\Omega$, and the following relation holds  
\begin{equation}\label{PL1}
(\Lambda_\varepsilon u^h) \nabla \Pi^h (F'_\varepsilon(u^h))=\nabla u^h \ \ \mbox{ in } \Omega.
\end{equation}
Basically, $\Lambda_\varepsilon u^h$ is a constant by elements matrix such that (\ref{PL1}) holds by elements. We highlight that (\ref{PL1}) is satisfied due to the right angled constraint requirement ({\bf H}) and the choice of $\mathbb{P}_1$-continuous FE for $U_h$. Moreover, the following stability estimate holds \cite{BB,GS}
\begin{equation}\label{estlambda}
\Vert \Lambda_\varepsilon (u^h)\Vert_{L^r}^r \leq C(1+\Vert u^h \Vert_1^2), \ \ \forall u^h \in U_h \ \mbox{ (for } r=2(d+1)/d),
\end{equation}
where the constant $C>0$ is independent of $\varepsilon$ and $h$. We recall the result below concerning to $\Lambda_\varepsilon(\cdot)$ (see \cite[Lemma 2.1]{BB}).
\begin{lem}\label{lemconv}
Let $\Vert \cdot \Vert$ denote the spectral norm on $\mathbb{R}^{d\times d}$. Then for any given $\varepsilon\in (0,1)$ the function $\Lambda_\varepsilon:U_h\rightarrow [L^\infty(\Omega)]^{d\times d}$ is continuous and satisfies
\begin{equation}\label{D}
\varepsilon \xi^T \xi \leq \xi^T \Lambda_\varepsilon(u^h) \xi \leq \varepsilon^{-1} \xi^T \xi, \ \ \forall \xi \in \mathbb{R}^d, \ \forall u^h\in U_h.
\end{equation}
In particular, for all $u^h_1,u^h_2 \in U_h$ and $K\in \mathcal{T}_h$ with vertices $\{p_l^K\}_{l=0}^d$, it holds
\begin{equation}\label{eqL}
\Vert (\Lambda_\varepsilon(u^h_1) - \Lambda_\varepsilon(u^h_2))\vert_K \Vert \leq \varepsilon^{-2} \max_{l=1,...,d} \{\vert u^h_1( p_{l}^K) - u^h_2 (p_{l}^K))\vert + \vert u^h_1( p_{0}^K) - u^h_2 (p_{0}^K))\vert  \},
\end{equation}
where $p^K_0$ is a fixed acute-angled vertex.
\end{lem}

Let $A_h: V_h \rightarrow V_h$ be the linear operator defined as follows
\begin{equation*}
(A_h v^h, \bar{v})=(\nabla v^h,\nabla \bar{v})+( v^h, \bar{v}) , \ \ \forall \bar{v}\in V_h.
\end{equation*}
Then, the following estimate holds (see for instance, \cite[Theorem 3.2]{FMD2}):
\begin{equation}\label{estW16}
\Vert v^h \Vert_{W^{1,6}}\leq C \Vert A_h v^h\Vert_0, \ \ \forall v^h \in V_h.
\end{equation}
Taking into account the regularized problem (\ref{modelf02acont}), we consider the following first order in time, nonlinear and coupled scheme: 
\begin{itemize}
\item{\underline{\emph{Scheme \textbf{UV}}:}\\
{\bf Initialization}: Let $(u^{0}_h,{v}_h^{0})=(Q^h u_0,R^h {v}_{0})\in  U_h\times V_h$. \\
{\bf Time step} n: Given $(u^{n-1}_\varepsilon,{v}^{n-1}_\varepsilon)\in  U_h\times {V}_h$, compute $(u^n_\varepsilon,{v}^{n}_\varepsilon)\in U_h \times {V}_h$ solving
\begin{equation}
\left\{
\begin{array}
[c]{lll}%
(\delta_t u^n_\varepsilon,\bar{u})^h + (\nabla u^n_\varepsilon,\nabla \bar{u}) + (\Lambda_\varepsilon (u^{n}_\varepsilon)\nabla {v}^n_\varepsilon,\nabla \bar{u})= 0, \ \ \forall \bar{u}\in U_h,\\
(\delta_t {v}^n_\varepsilon,\bar{v}) +(A_h  v^n_\varepsilon, \bar{v}) - (u^{n}_\varepsilon ,\bar{v}) =0,\ \ \forall
\bar{v}\in V_h,
\end{array}
\right.  \label{modelf02a}
\end{equation}}
where, in general, we denote $\delta_t a^n:= \displaystyle\frac{a^n - a^{n-1}}{k}$.
\end{itemize}

\subsection{Mass conservation and Energy-stability}\label{ELuv}
Since $\bar{u}=1\in U_h$ and $\bar{v}=1 \in V_h$, we deduce that the scheme \textbf{UV} is conservative in $u^n_\varepsilon$, that is,
\begin{equation}\label{conu1}
(u_\varepsilon^n,1)=(u^n_\varepsilon,1)^h= (u^{n-1}_\varepsilon,1)^h=\cdot\cdot\cdot= (u^{0}_h,1)^h=(u_h^0,1)=(Q^hu_0,1)=(u_0,1):=m_0,
\end{equation}
and we have the following behavior for $\int_\Omega v^n_\varepsilon$: 
\begin{equation}\label{conv1} 
\delta_t \left(\int_\Omega v^n_\varepsilon \right) + \int_\Omega v^n_\varepsilon= \int_\Omega u^n_\varepsilon =m_0 .
\end{equation}

\begin{lem} {\bf(Estimate of $\left\vert \int_\Omega v^n_\varepsilon\right\vert$) }\label{LemvL1}
The following estimate holds 
\begin{equation}\label{weak02UV}
\left\vert \int_\Omega v^n_\varepsilon \right\vert\leq (1+k)^{-n}  \left\vert \int_\Omega v_0 \right\vert+ m_0, \ \ \ \forall n\geq 0.
\end{equation}
\end{lem}
\begin{proof}
From (\ref{conv1}) we have $(1+k) \left\vert \int_\Omega v^n_\varepsilon \right\vert - \left\vert \int_\Omega v^{n-1}_\varepsilon \right\vert \leq k\; m_0$, and therefore, applying Lemma \ref{tmaD} (for $\delta=1$ and $\beta=m_0$), we arrive at
$$
\left\vert \int_\Omega v^n_\varepsilon \right\vert \leq (1+k)^{-n}  \left\vert \int_\Omega v^0_\varepsilon \right\vert+ m_0= (1+k)^{-n}  \left\vert \int_\Omega R^h v_0 \right\vert+ m_0,
$$
which implies (\ref{weak02UV}).
\end{proof}

\begin{defi}
A numerical scheme with solution $(u^n_\varepsilon,v^n_\varepsilon)$ is called energy-stable with respect to the  energy 
\begin{equation}\label{Euv}
\mathcal{E}_\varepsilon^h(u,v)=(F_\varepsilon(u),1)^h + \frac{1}{2} \Vert \nabla {v}\Vert_0^2
\end{equation}
if this energy is time decreasing, that is, $\mathcal{E}_\varepsilon^h(u^n_\varepsilon,v^n_\varepsilon)\leq \mathcal{E}_\varepsilon^h(u^{n-1}_\varepsilon,v^{n-1}_\varepsilon)$ for all $n\geq 1$.
\end{defi}

\begin{tma} {\bf (Unconditional stability)} \label{estinc1uv}
The scheme \textbf{UV} is unconditional energy stable with respect to $\mathcal{E}_\varepsilon^h(u,v)$. In fact, if $(u^n_\varepsilon,{v}^n_\varepsilon)$ is a solution of \textbf{UV}, then the following discrete energy law holds
\begin{equation}\label{deluv}
\delta_t \mathcal{E}_\varepsilon^h(u^n_\varepsilon,v^n_\varepsilon)+ \varepsilon\frac{k}{2}\Vert \delta_t u^n_\varepsilon \Vert_0^2 + \frac{k}{2} \Vert \delta_t \nabla v^n_\varepsilon\Vert_0^2 + \varepsilon\Vert \nabla u^n_\varepsilon\Vert_0^2 +\Vert (A_h - I) v^n_\varepsilon\Vert_0^2 + \Vert \nabla v^n_\varepsilon\Vert_0^2 \leq 0.
\end{equation}
\end{tma}
\begin{proof}
Testing (\ref{modelf02a})$_1$ by $\bar{u}= \Pi^h (F'_\varepsilon (u^n_\varepsilon))$ and (\ref{modelf02a})$_2$ by $\bar{v}=(A_h-I) v^n_\varepsilon$, adding and taking into account that $\Lambda_\varepsilon(u^n_\varepsilon)$ is symmetric as well as (\ref{PL1}) (which implies that $\nabla \Pi^h (F'_\varepsilon (u^n_\varepsilon))=\Lambda_\varepsilon^{-1} (u^{n}_\varepsilon) \nabla u^n_\varepsilon$), the terms $-(\Lambda_\varepsilon (u^{n}_\varepsilon) \nabla {v}^n_\varepsilon,\nabla \Pi^h (F'_\varepsilon (u^n_\varepsilon)))=-(\nabla {v}^n_\varepsilon,\Lambda_\varepsilon (u^{n}_\varepsilon)\nabla \Pi^h (F'_\varepsilon (u^n_\varepsilon)))=-(\nabla {v}^n_\varepsilon,\nabla u^{n}_\varepsilon)$ and $(u^n_\varepsilon,(A_h-I) v^n_\varepsilon)=(\nabla {u}^n_\varepsilon,\nabla v^{n}_\varepsilon)$ cancel, and we obtain
\begin{eqnarray}\label{I01a}
&(\delta_t u^n_\varepsilon,F'_\varepsilon (u^n_\varepsilon))^h&\!\!\! + \int_\Omega (\nabla u^n_\varepsilon)^T\!\cdot\!\Lambda_\varepsilon^{-1} (u^{n}_\varepsilon)\!\cdot\! \nabla u^n_\varepsilon d\x\nonumber\\
&&\!\!\! +
\delta_t \Big( \frac{1}{2} \Vert \nabla {v}^n_\varepsilon\Vert_0^2\Big) + \frac{k}{2} \Vert \delta_t \nabla v^n_\varepsilon\Vert_0^2 +\Vert (A_h-I) v^n_\varepsilon\Vert_0^2 + \Vert \nabla v^n_\varepsilon\Vert_0^2 = 0.
\end{eqnarray}
Moreover, observe that from the Taylor formula we have
\begin{equation*}
F_\varepsilon(u^{n-1}_\varepsilon)=F_\varepsilon(u^n_\varepsilon)+F'_\varepsilon(u^n_\varepsilon)(u^{n-1}_\varepsilon-u^n_\varepsilon)+\frac{1}{2}F''_\varepsilon(\theta
u^n_\varepsilon+(1-\theta)u^{n-1}_\varepsilon)(u^{n-1}_\varepsilon-u^n_\varepsilon)^2,
\end{equation*}
and therefore,
\begin{equation}\label{I01b}
F'_\varepsilon (u^n_\varepsilon) \delta_t u^n_\varepsilon = \delta_t \Big( F_\varepsilon(u^n_\varepsilon) \Big) + \frac{k}{2} F''_\varepsilon(\theta
u^n_\varepsilon+(1-\theta)u^{n-1}_\varepsilon)(\delta_t u^n_\varepsilon)^2.
\end{equation}
Then, using (\ref{I01b}) and taking into account that $\Pi^h$ is linear and $F''_\varepsilon(s)\geq \varepsilon$ for all $s\in \mathbb{R}$, we have
\begin{eqnarray}\label{I01c}
&(\delta_t u^n_\varepsilon,F'_\varepsilon (u^n_\varepsilon))^h&\!\!\!=\int_\Omega \Pi^h(\delta_t u^n_\varepsilon\cdot F'_\varepsilon (u^n_\varepsilon)) \nonumber\\
&&\!\!\!= \delta_t \Big( \int_\Omega \Pi^h (F_\varepsilon(u^n_\varepsilon)) \Big) + \frac{k}{2} \int_\Omega \Pi^h(F''_\varepsilon(\theta
u^n_\varepsilon+(1-\theta)u^{n-1}_\varepsilon)(\delta_t u^n_\varepsilon)^2)\nonumber\\
&& \!\!\! \geq \delta_t (F_\varepsilon(u^n_\varepsilon),1)^h  + \varepsilon\frac{k}{2}\vert \delta_t u^n_\varepsilon \vert_h^2.
\end{eqnarray}
Thus, from (\ref{D}), (\ref{I01a}), (\ref{I01c}) and Remark \ref{eqh2}, we arrive at (\ref{deluv}).
\end{proof}

\begin{cor} \label{welem} {\bf(Uniform estimates) }
Assume that $(u_0,v_0)\in L^2(\Omega)\times H^1(\Omega)$. Let $(u^n_\varepsilon,v^n_\varepsilon)$ be a solution of scheme \textbf{UV}. Then, it holds 
\begin{equation}\label{weak01uv-a}
(F_\varepsilon(u^n_\varepsilon),1)^h + \frac{1}{2}\Vert v^n_\varepsilon\Vert_{1}^{2}+k \underset{m=1}{\overset{n}{\sum}}\left( \varepsilon\Vert \nabla u^m_\varepsilon\Vert_0^2 +\Vert (A_h-I) v^m_\varepsilon\Vert_0^2 + \Vert \nabla v^m_\varepsilon\Vert_0^2\right)\leq C_0, \ \ \forall n\geq 1,
\end{equation}
\begin{equation}\label{weak01uv-b}
k \underset{m=n_0 + 1}{\overset{n+n_0}{\sum}} \Vert v^m_\varepsilon\Vert_{W^{1,6}}^2 \leq C_1(1+kn), \ \ \forall n\geq 1,
\end{equation}
where the integer $n_0\geq 0$ is arbitrary, with the constants $C_0,C_1>0$ depending on the data $(\Omega,u_0, v_0)$, but independent of $k,h,n$ and $\varepsilon$. Moreover, if $\varepsilon\in (0,e^{-2})$, the following estimates hold
\begin{equation}\label{pneg1-a}
 \int_\Omega (\Pi^h (u^n_{\varepsilon-}))^2 \leq C_0 \varepsilon, \ \ \mbox{ and} \ \
\int_\Omega \vert u^n_\varepsilon\vert \leq m_0+C\sqrt{\varepsilon}, \ \ \forall n\geq 1,
\end{equation}
where $u^n_{\varepsilon-}:=\min\{u^n_\varepsilon, 0 \}\leq 0$ and the constant $C>0$ depends on the data $(\Omega,u_0, v_0)$, but is independent of $k,h,n$ and $\varepsilon$.
\end{cor}
\begin{proof}
First, using the inequality $s(ln\; s - 1)\leq s^2$ for all $s>0$ (which implies $F_\varepsilon(s)\leq C(s^2 + 1)$ for all $s\geq 0$) and taking into account that $(u^{0}_h,{v}_h^{0})=(Q^h u_0,R^h {v}_{0})$, $u_0\geq 0$ (and therefore, $u^{0}_h\geq 0$), as well as the definition of $F_\varepsilon$, we have
\begin{eqnarray}\label{ee5}
&\mathcal{E}^h_\varepsilon(u^0_h,v^0_h)&\!\!\!= \int_\Omega \Pi^h (F_\varepsilon(u^0_h)) +  \frac{1}{2}\Vert \nabla v^0_h\Vert_{0}^{2} \leq C\int_\Omega \Pi^h ((u^0_h)^2 + 1) +  \frac{1}{2}\Vert \nabla v^0_h\Vert_{0}^{2}\nonumber\\
&&\!\!\! \leq C(\Vert u^0_h\Vert_0^2 + \Vert \nabla v^0_h\Vert_{0}^{2} + 1) \leq C(\Vert u_0\Vert_0^2 + \Vert v_0\Vert_{1}^{2} + 1)\leq C_0,
\end{eqnarray}
with the constant $C_0>0$ depending on the data $(\Omega,u_0, v_0)$, but independent of $k,h,n$ and $\varepsilon$. Therefore,  from the discrete energy law (\ref{deluv}) and (\ref{ee5}), we have
\begin{equation}\label{nolin2}
\mathcal{E}^h_\varepsilon(u^n_\varepsilon,v^n_\varepsilon)+k \underset{m=1}{\overset{n}{\sum}}\left( \varepsilon\Vert \nabla u^m_\varepsilon\Vert_0^2 +\Vert (A_h-I) v^m_\varepsilon\Vert_0^2 + \Vert \nabla v^m_\varepsilon\Vert_0^2\right) \leq  \mathcal{E}^h_\varepsilon(u^0_h,v^0_h)\leq C_0. \ \ \ \ \
\end{equation}
Thus, from (\ref{weak02UV}) and (\ref{nolin2}) we conclude (\ref{weak01uv-a}). Moreover, adding (\ref{deluv}) from $m=n_0+1$ to $m=n+n_0$, and using (\ref{estW16}) and (\ref{weak01uv-a}), we deduce (\ref{weak01uv-b}). By other hand, if $\varepsilon\in (0,e^{-2})$, from (\ref{PNa})$_2$ and taking into account that $F_\varepsilon (s)\geq 0$ for all $s\in \mathbb{R}$, we have $\frac{1}{2\varepsilon} (u^n_{\varepsilon-}(\x))^2 \leq F_\varepsilon (u^n_\varepsilon(\x))$ for all $u^n_\varepsilon \in U_h$; and therefore, using that $(\Pi^h (u))^2\leq \Pi^h(u^2)$ for all $u\in C(\overline{\Omega})$,  we have
\begin{equation*}
\frac{1}{2\varepsilon}\int_\Omega (\Pi^h (u^n_{\varepsilon-}))^2 \leq \frac{1}{2\varepsilon}\int_\Omega \Pi^h ((u^n_{\varepsilon-})^2) \leq  \int_\Omega \Pi^h (F_\varepsilon(u^n_\varepsilon)) \leq C_0,
\end{equation*}
where in the last inequality (\ref{weak01uv-a}) was used. Thus, we obtain (\ref{pneg1-a})$_1$. Finally, considering $u^n_{\varepsilon+}:=\max\{u^n_\varepsilon, 0 \}\geq 0$, taking into account that $ u^n_\varepsilon = u^n_{\varepsilon+} + u^n_{\varepsilon-}$ and $\vert u^n_\varepsilon\vert = u^n_{\varepsilon+} - u^n_{\varepsilon-}= u^n_{\varepsilon} - 2u^n_{\varepsilon-}$, using the H\"older and Young inequalities as well as (\ref{conu1}) and (\ref{pneg1-a})$_1$, we have
\begin{equation*}
\int_\Omega \vert u^n_\varepsilon\vert \leq \int_\Omega \Pi^h \vert u^n_\varepsilon\vert =  \int_\Omega u^n_\varepsilon - 2 \int_\Omega \Pi^h (u^n_{\varepsilon-}) \leq m_0 + C\Big( \int_\Omega (\Pi^h (u^n_{\varepsilon-}))^2\Big)^{1/2} \leq m_0+C\sqrt{\varepsilon},
\end{equation*}
which implies (\ref{pneg1-a})$_2$.
\end{proof}
\begin{obs}\label{RUe}
The $l^\infty(L^1)$-norm is the only norm in which $u^n_\varepsilon$ is bounded independently of $(k,h)$ and $\varepsilon$ (see (\ref{pneg1-a})$_2$).  However, we can also obtain $\varepsilon$-dependent bounds for $u^n_\varepsilon$. In fact, from (\ref{PNa}) and taking into account that $\varepsilon\in (0,e^{-2})$, we can deduce that $\displaystyle\frac{\varepsilon}{2}s^2\leq F_\varepsilon(s) + 2$ for all $s\in \mathbb{R}$, which together with (\ref{weak01uv-a}), implies that $(\sqrt{\varepsilon}\; u^n_\varepsilon)$ is bounded in $l^\infty(L^2) \cap l^2(H^1)$. 
\end{obs}
\begin{obs}\label{NNuh}{\bf (Approximated positivity)}
\begin{enumerate}
\item From (\ref{pneg1-a})$_1$, the following estimate holds
$$\max_{n\geq 0} \Vert \Pi^h (u^n_{\varepsilon-})\Vert_0^2 \leq C_0\varepsilon.$$ 
\item Assuming $V_h$ furnished by $\mathbb{P}_1$-continuous FE and considering the following approximation for the $v$-equation:
\begin{equation}\label{VV}
(\delta_t {v}^n_\varepsilon,\bar{v})^h +(\widetilde{A}_h  v^n_\varepsilon, \bar{v})^h - (u^{n}_\varepsilon ,\bar{v})^h =0,\ \ \forall
\bar{v}\in V_h,
\end{equation}
where $\widetilde{A}_h:V_h \rightarrow V_h$ is the operator defined by $(\widetilde{A}_h  v_h, \bar{v})^h=(\nabla v_h,\nabla \bar{v}) + (v_h,\bar{v})^h$ for all $\bar{v}\in V_h$, then the unconditional energy-stability also holds and the following estimates are satisfied 
\begin{equation}\label{mm1}
\max_{n\geq 0} \Vert \Pi^h (v^n_{\varepsilon-})\Vert_0^2 \leq C\varepsilon \ \ \mbox{ and } \ \  k \underset{m=1}{\overset{n}{\sum}}  \Vert \Pi^h (v^n_{\varepsilon-})\Vert_1^2\leq C \varepsilon (kn),
\end{equation}
where the constant $C$ is independent of $k,h,n$ and $\varepsilon$. In fact, testing by $\bar{v}=\Pi^h (v^n_{\varepsilon-}) \in V_h$ in (\ref{VV}), taking into account that $(\nabla \Pi^h (v^n_{\varepsilon+}), \nabla \Pi^h (v^n_{\varepsilon-}))\geq 0$ (owing to the interior angles of the triangles or tetrahedra are less than or equal to $\pi/2$), and using again that $(\Pi^h (v))^2\leq \Pi^h(v^2)$ for all $v\in C(\overline{\Omega})$, we have  
\begin{eqnarray*}
&\displaystyle\Big(\frac{1}{k} + 1 \Big) \Vert \Pi^h (v^n_{\varepsilon-})\Vert_0^2&\!\!\! + \Vert \nabla \Pi^h (v^n_{\varepsilon-})\Vert_0^2 \leq \int_\Omega \Pi^h\left[ \Big(u^n_\varepsilon + \frac{1}{k} v^{n-1}_\varepsilon\Big) v^n_{\varepsilon-}\right]\nonumber\\
&&\!\!\!\leq \displaystyle\int_\Omega \Pi^h\left[ \Big(u^n_{\varepsilon-} + \frac{1}{k} v^{n-1}_{\varepsilon-}\Big) v^n_{\varepsilon-}\right]\nonumber\\
&& \!\!\! \leq \displaystyle\frac{1}{2}\left(\frac{1}{k}+1\right) \Vert \Pi^h (v^n_{\varepsilon-})\Vert_0^2 + \frac{1}{2} \Vert \Pi^h (u^n_{\varepsilon-})\Vert_0^2 + \frac{1}{2k} \Vert \Pi^h (v^{n-1}_{\varepsilon-})\Vert_0^2,
\end{eqnarray*}
from which, using (\ref{pneg1-a})$_1$,  we arrive at
\begin{equation}\label{mm2}
\displaystyle\frac{1}{2}\Big(\frac{1}{k} + 1 \Big) \Vert \Pi^h (v^n_{\varepsilon-})\Vert_0^2+ \Vert \nabla \Pi^h (v^n_{\varepsilon-})\Vert_0^2 \leq \frac{1}{2}C_0\varepsilon + \frac{1}{2k} \Vert \Pi^h (v^{n-1}_{\varepsilon-})\Vert_0^2.
\end{equation}
Therefore, if $v^0_h\geq 0$ (which holds for instance by considering $v^0_h=\widetilde{R}^h v_0$, where $\widetilde{R}^h$  is an average interpolator of Clement or Scott-Zhang type, and using that $v_0\geq 0$), using Lemma \ref{tmaD} in (\ref{mm2}), we conclude (\ref{mm1})$_1$. Finally, multiplying (\ref{mm2}) by $k$ and adding from $m=1$ to $m=n$, and using again that ${v}_h^{0}\geq 0$, we arrive at (\ref{mm1})$_2$.
 \end{enumerate}
\end{obs}

\subsection{Well-posedness}
In this subsection, we will prove the well-posedness of the scheme \textbf{UV}. We recall that, taking into account that we remain in finite dimension, all norms are equivalents.
\begin{tma}\label{UVsolv} {\bf (Unconditional existence)} 
There exists at least one solution $(u^n_\varepsilon,v^n_\varepsilon)$ of the scheme \textbf{UV}.
\end{tma}
\begin{proof}
We will use the Leray-Schauder fixed point theorem. With this aim, given $(u^{n-1}_\varepsilon,v^{n-1}_\varepsilon)\in U_h\times V_h$, we define the operator $R:U_h\times V_h\rightarrow U_h\times V_h$ by  $R(\widetilde{u},\widetilde{v})=(u,{v})$, such that $(u,v)\in U_h\times V_h$ solves the following linear decoupled problem
\begin{equation}\label{modelfexis01-uvA}
u\in U_h \ \ \mbox{s.t. } \ \displaystyle\frac{1}{k}(u,\bar{u})^h + (\nabla u, \nabla\bar{u}) =\displaystyle\frac{1}{k}(u^{n-1}_\varepsilon,\bar{u})^h -(\Lambda_\varepsilon(\widetilde{u})\nabla \widetilde{v},\nabla \bar{u}), \ \ \forall \bar{u}\in U_h,
\end{equation}
\begin{equation}\label{modelfexis01-uvB}
v\in V_h \ \ \mbox{s.t. } \ \displaystyle\frac{1}{k}(v,\bar{v}) + (A_h v, \bar{v}) =\displaystyle\frac{1}{k}(v^{n-1}_\varepsilon,\bar{v})
+(\widetilde{u}, \bar{v}), \ \ \forall \bar{v}\in V_h.
\end{equation}
\begin{enumerate}
\item{$R$ is well defined}. Applying the Lax-Milgram theorem to (\ref{modelfexis01-uvA}) and (\ref{modelfexis01-uvB}), we can deduce that, for each $(\widetilde{u},\widetilde{v})\in U_h\times V_h$, there exists a unique $(u,{v})\in U_h \times V_h$ solution of (\ref{modelfexis01-uvA})-(\ref{modelfexis01-uvB}).
\item{Let us now prove that all possible fixed points of $\lambda R$ (with $\lambda \in (0,1]$) are bounded.}
In fact, observe that if $(u,{v})$ is a fixed point of $\lambda R$, then $R(u,v) = (\frac{1}{\lambda} u, \frac{1}{\lambda}{v})$, and therefore $(u,{v})$ satisfies the coupled system
\begin{equation}\label{eq010}
\left\{
\begin{array}
[c]{lll}%
\displaystyle \displaystyle\frac{1}{k} (u,\bar{u})^h+(\nabla u, \nabla\bar{u})+\lambda(\Lambda_\varepsilon(u)\nabla v,\nabla \bar{u})=\displaystyle\frac{\lambda}{k}(u^{n-1}_\varepsilon,\bar{u})^h ,\ \ \forall \bar{u}\in U_h, \vspace{0,2 cm}\\
\displaystyle\frac{1}{k} (v,\bar{v}) + (A_h v,\bar{v}) - \lambda (u, \bar{v})=
\displaystyle\frac{\lambda}{k}(v^{n-1}_\varepsilon,\bar{v}), \ \ \forall \bar{v}\in  V_h.
\end{array}
\right. 
\end{equation}
Then, testing (\ref{eq010})$_1$ and (\ref{eq010})$_2$ by $\bar{u}= \Pi^h (F'_\varepsilon (u))$ and $\bar{v}=(A_h -I) v$ respectively, proceeding as in Theorem \ref{estinc1uv} and taking into account that $\lambda \in (0,1]$, we obtain
\begin{eqnarray}\label{pfac1}
&\displaystyle(F_\varepsilon(u),1)^h+ \frac{1}{2}\Vert \nabla v\Vert_{0}^{2}&\!\!\!\!+k \left(\varepsilon \Vert\nabla u\Vert_0^2 +\Vert (A_h-I) v\Vert_0^2 + \Vert \nabla v\Vert_0^2\right) \nonumber\\
&&\!\!\!\!\leq  (F_\varepsilon(\lambda u^{n-1}_\varepsilon),1)^h + \frac{\lambda^2}{2} \Vert \nabla {v}^{n-1}_\varepsilon\Vert_{0}^{2}\leq C(u^{n-1}_\varepsilon, {v}^{n-1}_\varepsilon),
\end{eqnarray}
where the last estimate is $\lambda$-independent (arguing as in (\ref{ee5})). Moreover, proceeding as in Lemma \ref{LemvL1} and Corollary \ref{welem} (taking into account (\ref{pfac1})), we deduce $\Vert (u,v)\Vert_{L^1\times H^1}\leq C$, where the constant $C$ depends on data $(\Omega, u^{n-1}_\varepsilon, {v}^{n-1}_\varepsilon,\varepsilon)$, but it is independent of $\lambda$ and $h$.

\item{We  prove that $R$ is continuous.} Let $\{(\widetilde{u}^l,\widetilde{v}^l)\}_{l\in\mathbb{N}}\subset U_h\times V_h\hookrightarrow W^{1,\infty}(\Omega)^2$ be a sequence such that 
\begin{equation}\label{c001-uv}
(\widetilde{u}^l,\widetilde{v}^l)\rightarrow (\widetilde{u},\widetilde{v}) \ \mbox{ in }  U_h\times V_h
\quad \hbox{as $l\to +\infty$}.
\end{equation}
In particular, since we remain in finite dimension, $\{(\widetilde{u}^l,\widetilde{v}^l)\}_{l\in\mathbb{N}}$ is bounded in $W^{1,\infty}(\Omega)^2$. Then, if we denote $(u^l,v^l)=R(\widetilde{u}^l,\widetilde{v}^l)$, we can deduce
\begin{eqnarray*}
\label{fps02}
&\displaystyle
\frac{1}{2k}\Vert (u^l, v^l)\Vert_{0}^{2}&\!\!\!+\frac{1}{2}\Vert \nabla u^l\Vert_{0}^{2} +\frac{1}{2}\Vert v^l \Vert_1^2 \nonumber\\
&&\!\!\!\!\!\!\!\!\!\!\!\!\!\leq \displaystyle
\frac{1}{2k}\Vert (u^{n-1}_\varepsilon, v^{n-1}_\varepsilon)\Vert_{0}^{2} +  (1 +\Vert \widetilde{u}^l\Vert_1^2)^{2/r} \Vert \nabla \widetilde{v}^l \Vert_{L^\infty}^2+ C\Vert \widetilde{u}^l\Vert_{0}^2 \leq C,
\end{eqnarray*}
where in the first inequality  (\ref{estlambda}) was used and  $C$ is a constant independent of $l\in\mathbb{N}$. Therefore, $\{(u^l,v^l)=R(\widetilde{u}^l,\widetilde{v}^l)\}_{l\in\mathbb{N}}$ is bounded in $U_h\times V_h\hookrightarrow W^{1,\infty}(\Omega)^2$. Then, there exists a subsequence of  $\{R(\widetilde{u}^l,\widetilde{\boldsymbol \sigma}^l)\}_{l\in\mathbb{N}}$, still denoted by$\{R(\widetilde{u}^{l},\widetilde{\boldsymbol \sigma}^{l})\}_{l\in\mathbb{N}}$,  such that 
\begin{equation}\label{c002-uv}
R(\widetilde{u}^{l},\widetilde{v}^{l})\rightarrow (u',{v}') \ \ \mbox{ in } \ W^{1,\infty}(\Omega)^2, \quad \hbox{as $l\to +\infty$}.
\end{equation}
Then, from (\ref{c001-uv})-(\ref{c002-uv}) and using Lemma \ref{lemconv}, a standard procedure allows us to pass to the limit, as $l$ goes to $+\infty$, in (\ref{modelfexis01-uvA})-(\ref{modelfexis01-uvB}) (with $(\widetilde{u}^l,\widetilde{v}^l)$ and $(u^l, v^l)$ instead of $(\widetilde{u},\widetilde{v})$ and $(u,v)$ respectively), and we deduce that $R(\widetilde{u},\widetilde{v})=({u}',{v}')$. Therefore, we have proved that any convergent subsequence of  $\{R(\widetilde{u}^l,\widetilde{v}^l)\}_{l\in\mathbb{N}}$ converges to $R(\widetilde{u},\widetilde{v})$ in $U_h\times V_h$, and from uniqueness of $R(\widetilde{u},\widetilde{v})$, we conclude that the whole sequence $R(\widetilde{u}^l,\widetilde{v}^l)\rightarrow R(\widetilde{u},\widetilde{v})$ in $U_h\times V_h$. Thus, $R$ is continuous.
\end{enumerate}
Therefore, the hypotheses of the Leray-Schauder fixed point theorem (in finite dimension) are satisfied and we conclude that the map $R$ has a fixed point $(u,v)$, that is
$R(u,{v})=(u,{v})$, which is a solution of the scheme \textbf{UV}. 
\end{proof}

\begin{lem}{\bf (Conditional uniqueness)}
If $k\, g(h,\varepsilon)<1$ (where $g(h,\varepsilon)\uparrow +\infty$ as $h\downarrow 0$ or $\varepsilon\downarrow 0$),  then the solution $(u^n_\varepsilon,v^n_\varepsilon)$ of the scheme \textbf{UV} is unique.
\end{lem}
\begin{proof}
Suppose that  there exist $(u^{n,1}_\varepsilon,v^{n,1}_\varepsilon),(u^{n,2}_\varepsilon,v^{n,2}_\varepsilon)\in U_h\times  V_h$ two possible solutions of the scheme \textbf{UV}. Then, defining $u=u^{n,1}_\varepsilon-u^{n,2}_\varepsilon$ and $v=v^{n,1}_\varepsilon-v^{n,2}_\varepsilon$,
we have that $(u,v)\in U_h\times  V_h$ satisfies, for all $(\bar{u},\bar{v})\in U_h\times V_h$,
\begin{equation}\label{uniq001v}
\displaystyle\frac{1}{k}(u,\bar{u})^h+ (\nabla u,
\nabla \bar{u})+(\Lambda_\varepsilon(u^{n,1}_\varepsilon) \nabla v,\nabla \bar{u})
+((\Lambda_\varepsilon(u^{n,1}_\varepsilon) - \Lambda_\varepsilon(u^{n,2}_\varepsilon))\nabla v^{n,2}_\varepsilon,\nabla \bar{u})=0,
\end{equation}
\begin{equation}\label{uniq002v}
\displaystyle\frac{1}{k}(v,\bar{v})+ (A_h v,\bar{v}) =(u,\bar{v}).
\end{equation}
Taking $\bar{u}=u$, $\bar{v}=A_h v$ in (\ref{uniq001v})-(\ref{uniq002v}), adding the resulting expressions and using the fact that $\displaystyle\int_\Omega u=0$ and the equivalence of the norms $\Vert \cdot \Vert_0$ and $\vert \cdot \vert_h$ in $U_h$ given in  Remark \ref{eqh2}, we obtain
\begin{eqnarray*}
&\displaystyle\frac{1}{k}&\!\!\!\!\Vert (u,\nabla v)\Vert_{0}^2+ \Vert (u,A_h v)\Vert_{H^1\times L^2}^2 \leq \Vert u\Vert_0\Vert A_h v\Vert_0 \nonumber\\
&&+\Vert \Lambda_\varepsilon(u^{n,1}_\varepsilon)\Vert_{L^6} \Vert \nabla v\Vert_{L^3} \Vert \nabla u\Vert_0 + \Vert \Lambda_\varepsilon(u^{n,1}_\varepsilon) - \Lambda_\varepsilon(u^{n,2}_\varepsilon)\Vert_{L^\infty} \Vert \nabla
v^{n,2}_\varepsilon\Vert_{0} \Vert \nabla u\Vert_{0}\nonumber\\
&&\leq \frac{1}{4} \Vert A_h v\Vert_0 + \Vert u\Vert_0^2 + \frac{1}{4} \Vert \nabla u\Vert_0^2 + \frac{1}{4} \Vert A_h v\Vert_{0}^2 + C \Vert \Lambda_\varepsilon(u^{n,1}_\varepsilon)\Vert_{L^6}^4 \Vert \nabla v\Vert_{0}^2  \nonumber\\
&&+ \frac{1}{4} \Vert \nabla u\Vert_0^2 +  \Vert \Lambda_\varepsilon(u^{n,1}_\varepsilon) - \Lambda_\varepsilon(u^{n,2}_\varepsilon)\Vert_{L^\infty}^2 \Vert \nabla
v^{n,2}_\varepsilon\Vert^2_{0}.
\end{eqnarray*}
Then, taking into account (\ref{estlambda}), (\ref{eqL}), (\ref{weak01uv-a}), (\ref{pneg1-a})$_2$ and using the inverse inequalities: $\Vert u^h \Vert_{L^6}^2\leq C_1(h) \Vert u^h \Vert_{L^r}^2$, $\Vert u^h \Vert_{1}^2\leq C_2(h) \Vert u^h \Vert_{L^1}^2$ and $\Vert u^h \Vert_{L^\infty}^2\leq C_3(h) \Vert u^h \Vert_0^2$ for all $u^h \in U_h$, we have
\begin{eqnarray*}\label{uniq003}
&\displaystyle\Vert &\!\!\!\!\!(u,\nabla v)\Vert_{0}^2+ \frac{k}{2}\Vert (u,A_h v)\Vert_{H^1\times L^2}^2\leq k 
\left(1+ C \Vert \Lambda_\varepsilon(u^{n,1}_\varepsilon)\Vert_{L^6}^4\right)\Vert (u,\nabla v) \Vert_0^2 + k C_0 C\varepsilon^{-2} \Vert u\Vert_{L^\infty}^2\nonumber\\
&& \!\!\! \leq  k 
\left(1+ C_1(h)^2(1 + C_2(h))^{4/r}+  k C_0 C_3(h)\varepsilon^{-2} \right)\Vert (u,\nabla v) \Vert_0^2:= k\,g(h,\varepsilon)\Vert (u,\nabla v) \Vert_0^2.
\end{eqnarray*}
Therefore, if $k\,g(h,\varepsilon)< 1$, we conclude that $u=0$, and therefore (from (\ref{uniq002v})) $v=0$.
\end{proof}

\section{Scheme US}
In this section, we propose another energy-stable nonlinear fully discrete scheme associated to model (\ref{modelf00}), which is obtained by introducing the auxiliary variable ${\boldsymbol\sigma}=\nabla v$. In fact, taking into account the functions $\lambda_\varepsilon$ and $F_\varepsilon$ and its derivatives (given in (\ref{aE})-(\ref{F2pE})), another regularized version of problem (\ref{modelf00}) reads: Find $u_\varepsilon:\Omega\times [0,T]\rightarrow \mathbb{R}$ and ${\boldsymbol\sigma}_\varepsilon:\Omega\times [0,T]\rightarrow \mathbb{R}^d$ such that
\begin{equation}
\left\{
\begin{array}
[c]{lll}%
\partial_t u_\varepsilon -\nabla \cdot(\lambda_\varepsilon (u_\varepsilon) \nabla (F'_\varepsilon(u_\varepsilon))) -\nabla\cdot( u_\varepsilon {\boldsymbol\sigma}_\varepsilon)=0\ \ \mbox{in}\ \Omega,\ t>0,\\
\partial_t {\boldsymbol \sigma}_\varepsilon + \mbox{rot(rot } {\boldsymbol \sigma}_\varepsilon \mbox{)} -\nabla(\nabla \cdot {\boldsymbol \sigma}_\varepsilon) + {\boldsymbol \sigma}_\varepsilon = u_\varepsilon\nabla (F'_\varepsilon(u_\varepsilon))\ \ \mbox{in}\ \Omega,\ t>0,\\
\displaystyle \frac{\partial u_\varepsilon}{\partial \mathbf{n}}=0\ \ \mbox{on}\ \partial\Omega,\ t>0,\\
{\boldsymbol \sigma}_\varepsilon\cdot \mathbf{n}=0, \ \ \left[\mbox{rot }{\boldsymbol \sigma}_\varepsilon \times \mathbf{n}\right]_{tang}=0 \quad \mbox{on}\ \partial\Omega,\ t>0,\\
u_\varepsilon(\x ,0)=u_0(\x )\geq 0,\ {\boldsymbol \sigma}_\varepsilon(\x ,0)=\nabla v_0(
\x ),\quad \mbox{in}\ \Omega.
\end{array}
\right.  \label{modelf02acontUS}
\end{equation}
This kind of formulation considering ${\boldsymbol\sigma}=\nabla v$ as auxiliary variable has been used in the construction of numerical schemes for other chemotaxis models (see for instance \cite{FMD2,Z1}). Once problem (\ref{modelf02acontUS}) is solved, we can recover $v_\varepsilon$ from $u_\varepsilon$ solving
\begin{equation} \left\{
\begin{array}
[c]{lll}%
\partial_t v_\varepsilon -\Delta v_\varepsilon + v_\varepsilon = u_\varepsilon  \quad \mbox{in}\ \Omega,\ t>0,\\
\displaystyle
\frac{\partial v_\varepsilon}{\partial \mathbf{n}}=0\quad\mbox{on}\ \partial\Omega,\ t>0,\\
 v_\varepsilon(\x ,0)=v_0(\x )\geq 0\quad \mbox{in}\ \Omega.
\end{array}
\right.  \label{modelf01eqv}
\end{equation}
Observe that multiplying (\ref{modelf02acontUS})$_1$ by $F'_\varepsilon(u_\varepsilon)$, (\ref{modelf02acontUS})$_2$ by ${\boldsymbol\sigma}_\varepsilon$, and integrating over $\Omega$, we obtain the following energy law
\begin{eqnarray*}
\frac{d}{dt} \displaystyle \int_\Omega \Big( F_\varepsilon(u_\varepsilon) + \frac{1}{2} \vert {\boldsymbol\sigma}_\varepsilon\vert^2\Big) d \x + \int_\Omega \lambda_\varepsilon(u_\varepsilon)\vert \nabla (F'_\varepsilon(u_\varepsilon))\vert^2 d\x+\Vert {\boldsymbol\sigma}_\varepsilon\Vert_1^2 = 0.
\end{eqnarray*}
In particular, the modified energy $\mathcal{E}_\varepsilon(u,{\boldsymbol\sigma})= \displaystyle\int_\Omega \Big( F_\varepsilon(u) + \frac{1}{2} \vert {\boldsymbol\sigma}\vert^2\Big) d \x$ is decreasing in time. 
Then, we consider a fully discrete approximation of the regularized problem (\ref{modelf02acontUS}) using a FE discretization in space and the backward
Euler discretization in time (again considered for simplicity on a uniform partition of $[0,T]$ with time step
$k=T/N : (t_n = nk)_{n=0}^{n=N}$). Concerning the space discretization, we consider the triangulation as in the scheme \textbf{UV}, but in this case without imposing the constraint ({\bf H}) related with the right-angles simplices. We choose the following continuous FE spaces for $u_\varepsilon$, ${\boldsymbol\sigma}_\varepsilon$, and $v_\varepsilon$:
$$(U_h,{\boldsymbol\Sigma}_h, V_h) \subset H^1(\Omega)^3,\quad \hbox{generated by $\mathbb{P}_1,\mathbb{P}_{m},\mathbb{P}_r$ with $m,r\geq 1$.}
$$
Then, we consider the following first order in time, nonlinear and coupled scheme: 
\begin{itemize}
\item{\underline{\emph{Scheme \textbf{US}}:}\\
{\bf Initialization}:  Let $(u_h^{0},{\boldsymbol \sigma}_h^{0})=(Q^h u_0, \widetilde{Q}^h (\nabla v_0))\in  U_h\times {\boldsymbol\Sigma}_h$.\\
{\bf Time step} n: Given $(u^{n-1}_\varepsilon,{\boldsymbol \sigma}^{n-1}_\varepsilon)\in  U_h\times {\boldsymbol\Sigma}_h$, compute $(u^{n}_\varepsilon,{\boldsymbol \sigma}^{n}_\varepsilon)\in  U_h \times {\boldsymbol\Sigma}_h$ solving
\begin{equation}
\left\{
\begin{array}
[c]{lll}%
(\delta_t u^n_\varepsilon,\bar{u})^h + (\lambda_\varepsilon (u^{n}_\varepsilon) \nabla \Pi^h(F'_\varepsilon(u^n_\varepsilon)),\nabla \bar{u}) = -(\lambda_\varepsilon (u^{n}_\varepsilon) {\boldsymbol\sigma}^n_\varepsilon,\nabla \bar{u}), \ \ \forall \bar{u}\in U_h,\\
(\delta_t {\boldsymbol \sigma}^n_\varepsilon,\bar{\boldsymbol \sigma}) + ( B_h {\boldsymbol \sigma}^n_\varepsilon,\bar{\boldsymbol \sigma}) =
(\lambda_\varepsilon (u^{n}_\varepsilon) \nabla  \Pi^h(F'_\varepsilon(u^n_\varepsilon)),\bar{\boldsymbol \sigma}),\ \ \forall
\bar{\boldsymbol \sigma}\in \Sigma_h,
\end{array}
\right.  \label{modelf02aUS}
\end{equation}}
\end{itemize}
where $Q^h$ is the $L^2$-projection on $U_h$ defined in (\ref{MLP2}), $\widetilde{Q}^h$ the standard $L^2$-projection on ${\boldsymbol\Sigma}_h$, and the operator $B_h$ is defined as 
$$(B_h {\boldsymbol \sigma}^n_\varepsilon,\bar{\boldsymbol \sigma}) = (\mbox{rot } {\boldsymbol\sigma}_\varepsilon^n,\mbox{rot } \bar{\boldsymbol\sigma}) + (\nabla \cdot {\boldsymbol\sigma}_\varepsilon^n, \nabla \cdot \bar{\boldsymbol\sigma}) + ({\boldsymbol\sigma}_\varepsilon^n, \bar{\boldsymbol\sigma}).$$
We recall that $\Pi^h: C(\overline{\Omega})\rightarrow U_h$ is the Lagrange interpolation operator, and the discrete semi-inner product $(\cdot,\cdot)^h$ was defined in (\ref{mlump}). Once the scheme \textbf{US} is solved, given $v^{n-1}_\varepsilon\in V_h$, we can recover $v^n_\varepsilon=v^n_\varepsilon(u^n_\varepsilon) \in V_h$ solving: 
\begin{equation}\label{edovf}
(\delta_t v^n_\varepsilon,\bar{v})  +(\nabla v^n_\varepsilon,\nabla \bar{v}) + (v^n_\varepsilon,\bar{v}) =(u^n_\varepsilon,\bar{v}), \ \ \forall
\bar{v}\in V_h.
\end{equation}

Given $u^n_\varepsilon\in U_h$ and $v^{n-1}_\varepsilon\in V_h$, Lax-Milgram theorem implies that there exists a unique $v^n_\varepsilon \in V_h$ solution of (\ref{edovf}). The solvability of (\ref{modelf02aUS}) will be proved in Subsection \ref{SSus}.

\subsection{Mass conservation and Energy-stability}\label{ELus}
Observe that the scheme \textbf{US}  is also conservative in $u$ (satisfying \eqref{conu1}) and also has the  behavior for $\int_\Omega v_n$ given in \eqref{conv1}.

\begin{defi}
A numerical scheme with solution $(u^n_\varepsilon,{\boldsymbol\sigma}^n_\varepsilon)$ is called energy-stable with respect to the  energy 
\begin{equation}\label{Eus}
\mathcal{E}_\varepsilon^h(u,{\boldsymbol\sigma})=(F_\varepsilon(u),1)^h + \frac{1}{2} \Vert {\boldsymbol\sigma}\Vert_0^2
\end{equation}
if this energy is time decreasing, that is, $\mathcal{E}_\varepsilon^h(u^n_\varepsilon,{\boldsymbol\sigma}^n_\varepsilon)\leq \mathcal{E}_\varepsilon^h(u^{n-1}_\varepsilon,{\boldsymbol\sigma}^{n-1}_\varepsilon)$ for all $n\geq 1$.
\end{defi}

\begin{tma} {\bf (Unconditional stability)} \label{estinc1us}
The scheme \textbf{US} is unconditional energy stable with respect to $\mathcal{E}_\varepsilon^h(u,{\boldsymbol\sigma})$. In fact, if $(u^n_\varepsilon,{\boldsymbol\sigma}^n_\varepsilon)$ is a solution of \textbf{US}, then the following discrete energy law holds
\begin{equation}\label{delus}
\delta_t \mathcal{E}_\varepsilon^h(u^n_\varepsilon,{\boldsymbol\sigma}^n_\varepsilon) + \varepsilon\frac{k}{2}\Vert \delta_t u^n_\varepsilon\Vert_0^2 + \frac{k}{2} \Vert \delta_t  {\boldsymbol\sigma}^n_\varepsilon\Vert_0^2 + \int_\Omega \lambda_\varepsilon (u^{n}_\varepsilon)\vert \nabla\Pi^h(F'_\varepsilon(u^n_\varepsilon)) \vert^2 d\x +\Vert  {\boldsymbol\sigma}^n_\varepsilon\Vert_1^2 \leq 0.
\end{equation}
\end{tma}
\begin{proof}
Testing (\ref{modelf02aUS})$_1$ by $\bar{u}= \Pi^h (F'_\varepsilon (u^n_\varepsilon))$, (\ref{modelf02aUS})$_2$ by $\bar{\boldsymbol\sigma}={\boldsymbol\sigma}^n_\varepsilon$ and adding, the terms \\
$(\lambda_\varepsilon (u^{n}_\varepsilon) \nabla  \Pi^h(F'_\varepsilon(u_\varepsilon)),{\boldsymbol \sigma}^n_\varepsilon)$ cancel, and we obtain
\begin{equation*}
(\delta_t u^n_\varepsilon,\Pi^h (F'_\varepsilon (u^n_\varepsilon)))^h + \int_\Omega \lambda_\varepsilon (u^{n}_\varepsilon)\vert \nabla\Pi^h(F'_\varepsilon(u^n_\varepsilon)) \vert^2 d\x  +
\delta_t \Big(\frac{1}{2} \Vert {\boldsymbol\sigma}^n_\varepsilon\Vert_0^2\Big) + \frac{k}{2} \Vert \delta_t  {\boldsymbol\sigma}^n_\varepsilon\Vert_0^2 +\Vert  {\boldsymbol\sigma}^n_\varepsilon\Vert_1^2 = 0,
\end{equation*}
which, proceeding as in (\ref{I01b})-(\ref{I01c}) and using Remark \ref{eqh2}, implies (\ref{delus}).
\end{proof}

\begin{cor} \label{welemUS} {\bf(Uniform estimates) }
Assume that $(u_0,v_0)\in L^2(\Omega)\times H^1(\Omega)$. Let $(u^n_\varepsilon,{\boldsymbol\sigma}^n_\varepsilon)$ be a solution of scheme \textbf{US}. Then, it holds 
$$
(F_\varepsilon(u^n_\varepsilon),1)^h + \Vert {\boldsymbol\sigma}^n_\varepsilon\Vert_{0}^{2}+k \underset{m=1}{\overset{n}{\sum}}\left( \varepsilon\Vert \nabla\Pi^h(F'_\varepsilon(u^m_\varepsilon))\Vert_0^2 +\Vert  {\boldsymbol\sigma}^m_\varepsilon\Vert_1^2\right)\leq C_0, \ \ \forall n\geq 1,
$$
with the constant $C_0>0$ depending on the data $(\Omega, u_0, v_0)$, but independent of $k,h,n$ and $\varepsilon$. Moreover, if $\varepsilon\in (0,e^{-2})$, estimates (\ref{pneg1-a}) hold.
\end{cor}
\begin{proof}
Proceeding as in (\ref{ee5}) (using the fact that $(u_h^{0},{\boldsymbol \sigma}_h^{0})=(Q^h u_0, \widetilde{Q}^h (\nabla v_0))$), we can deduce that 
\begin{equation}\label{ee6}
(F_\varepsilon(u^0_h),1)^h+ \Vert {\boldsymbol\sigma}^0_h\Vert_{0}^{2}\leq C_0, \ \ \ \ \
\end{equation}
where $C_0>0$ is a constant depending on the data $(\Omega, u_0, v_0)$, but independent of $k,h,n$ and $\varepsilon$. Therefore, from the discrete energy law (\ref{delus}) and estimate (\ref{ee6}), we have
\begin{equation*}
(F_\varepsilon(u^n_\varepsilon),1)^h+ \Vert {\boldsymbol\sigma}^n_\varepsilon\Vert_{0}^{2}+k \underset{m=1}{\overset{n}{\sum}}\left( \varepsilon\Vert\nabla\Pi^h(F'_\varepsilon(u^m_\varepsilon))\Vert_0^2 +\Vert {\boldsymbol\sigma}^m_\varepsilon\Vert_1^2\right) \leq  (F_\varepsilon(u^0_h),1)^h+ \Vert {\boldsymbol\sigma}^0_h\Vert_{0}^{2}\leq C_0. \ \ \ \ \
\end{equation*}
Finally,  the estimates given in (\ref{pneg1-a}) are proved as in Corollary \ref{welem}.
\end{proof}
\begin{obs}\label{NNuhs}
The conclusions obtained in Remark \ref{RUe} and the approximated positivity results established in Remark \ref{NNuh} remain true for the scheme \textbf{US}.
\end{obs}

\subsection{Well-posedness}\label{SSus}
\begin{tma} {\bf (Unconditional existence)} 
There exists at least one solution $(u^n_\varepsilon,{\boldsymbol\sigma}^n_\varepsilon)$  of scheme \textbf{US}.
\end{tma}
\begin{proof}
We will use the Leray-Schauder fixed point theorem. With this aim, given $(u^{n-1}_\varepsilon,{\boldsymbol\sigma}^{n-1}_\varepsilon)\in U_h\times {\boldsymbol\Sigma}_h$, we define the operator $R:U_h\times {\boldsymbol\Sigma}_h\rightarrow U_h\times {\boldsymbol\Sigma}_h$ by 
$R(\widetilde{u},\widetilde{\boldsymbol\sigma})=(u,{\boldsymbol\sigma})$, such that $(u,{\boldsymbol\sigma})\in U_h\times {\boldsymbol\Sigma}_h$ solves the following linear decoupled problem
\begin{equation} \label{modelfexis01-usA}
u\in U_h \ \ \mbox{s.t. } \ \displaystyle\frac{1}{k}(u,\bar{u})^h =\displaystyle\frac{1}{k}(u^{n-1}_\varepsilon,\bar{u})^h- (\lambda_\varepsilon (\widetilde{u}) \nabla \Pi^h(F'_\varepsilon(\widetilde{u})),\nabla \bar{u})   -(\lambda_\varepsilon (\widetilde{u}) \widetilde{\boldsymbol\sigma},\nabla \bar{u}), \ \ \forall \bar{u}\in U_h,
\end{equation}
\begin{equation} \label{modelfexis01-usB}
{\boldsymbol \sigma}\in {\boldsymbol \Sigma}_h \ \ \mbox{s.t. } \ \displaystyle\frac{1}{k}({\boldsymbol \sigma},\bar{\boldsymbol\sigma}) + (B_h{\boldsymbol \sigma},\bar{\boldsymbol \sigma})=\displaystyle\frac{1}{k}({\boldsymbol \sigma}^{n-1}_\varepsilon,\bar{\boldsymbol\sigma})
+(\lambda_\varepsilon (\widetilde{u}) \nabla  \Pi^h(F'_\varepsilon(\widetilde{u})),\bar{\boldsymbol \sigma}),\ \ \forall
\bar{\boldsymbol \sigma}\in \Sigma_h.
\end{equation}
\begin{enumerate}
\item{$R$ is well defined}. Applying the Lax-Milgram theorem to (\ref{modelfexis01-usA}) and (\ref{modelfexis01-usB}), we can deduce that, for each $(\widetilde{u},\widetilde{\boldsymbol \sigma})\in U_h \times {\boldsymbol \Sigma}_h$, there exists a unique 
$(u,{\boldsymbol \sigma})\in U_h \times {\boldsymbol \Sigma}_h$ solution of (\ref{modelfexis01-usA})-(\ref{modelfexis01-usB}).
\item{Let us now prove that all possible fixed points of $\lambda R$ (with $\lambda \in (0,1]$) are bounded.}
In fact, observe that if $(u,{\boldsymbol \sigma})$ is a fixed point of $\lambda R$, then $(u,{\boldsymbol \sigma})$ satisfies the coupled system
\begin{equation}\label{eq010us}
\left\{
\begin{array}
[c]{lll}%
\displaystyle \displaystyle\frac{1}{k} (u,\bar{u})^h +\lambda (\lambda_\varepsilon ({u}) \nabla \Pi^h(F'_\varepsilon(u)),\nabla \bar{u}) +\lambda (\lambda_\varepsilon ({u}) \boldsymbol\sigma,\nabla \bar{u}) =\displaystyle\frac{\lambda}{k}(u^{n-1}_\varepsilon,\bar{u})^h ,\ \ \forall \bar{u}\in U_h, \vspace{0,2 cm}\\
\displaystyle\displaystyle\frac{1}{k}({\boldsymbol \sigma},\bar{\boldsymbol\sigma}) + ( B_h {\boldsymbol \sigma},\bar{\boldsymbol \sigma})- \lambda(\lambda_\varepsilon ({u}) \nabla  \Pi^h(F'_\varepsilon(u)),\bar{\boldsymbol \sigma}) =
\displaystyle\frac{\lambda}{k}(\boldsymbol\sigma^{n-1}_\varepsilon,\bar{\boldsymbol\sigma}), \ \ \forall \bar{\boldsymbol \sigma}\in  {\boldsymbol \Sigma}_h.
\end{array}
\right. 
\end{equation}
Then, testing (\ref{eq010us})$_1$ and (\ref{eq010us})$_2$ by $\bar{u}= \Pi^h (F'_\varepsilon (u))\in U_h$ and $\bar{\boldsymbol \sigma}={\boldsymbol \sigma}\in {\boldsymbol\Sigma}_h$ respectively, proceeding as in Theorem \ref{estinc1us} and taking into account that $\lambda \in (0,1]$, we obtain
\begin{eqnarray}\label{pfac1us}
&\displaystyle(F_\varepsilon(u),1)^h+ \frac{1}{2}\Vert {\boldsymbol\sigma}\Vert_{0}^{2}&\!\!\!\!+k \left(\varepsilon \lambda \Vert\nabla  \Pi^h(F'_\varepsilon(u))\Vert_0^2 +\Vert \boldsymbol\sigma\Vert_1^2\right) \nonumber\\
&&\!\!\!\!\leq  (F_\varepsilon(\lambda u^{n-1}_\varepsilon),1)^h + \frac{\lambda^2}{2} \Vert \boldsymbol\sigma^{n-1}_\varepsilon\Vert_{0}^{2}\leq C(u^{n-1}_\varepsilon, \boldsymbol\sigma^{n-1}_\varepsilon),
\end{eqnarray}
which implies $\Vert {\boldsymbol \sigma}\Vert_{1}\leq C$ (with the constant $C>0$ independent of $\lambda$). Moreover, proceeding as in the proof of (\ref{pneg1-a}) (using (\ref{pfac1us})) we deduce $\Vert u\Vert_{L^1}\leq C$, where the constant $C$ depends on data $(\Omega, u^{n-1}_\varepsilon, {\boldsymbol \sigma}^{n-1}_\varepsilon,\varepsilon)$.

\item{We  prove that $R$ is continuous.} Let $\{(\widetilde{u}^l,\widetilde{\boldsymbol \sigma}^l)\}_{l\in\mathbb{N}}\subset U_h\times {\boldsymbol \Sigma}_h\hookrightarrow W^{1,\infty}(\Omega)\times  \W^{1,\infty}(\Omega)$ be a sequence such that 
\begin{equation}\label{c001-us}
(\widetilde{u}^l,\widetilde{\boldsymbol \sigma}^l)\rightarrow (\widetilde{u},\widetilde{\boldsymbol \sigma}) \ \mbox{ in }  U_h\times {\boldsymbol \Sigma}_h
\quad \hbox{as $l\to +\infty$}.
\end{equation}
In particular, $\{(\widetilde{u}^l,\widetilde{\boldsymbol \sigma}^l)\}_{l\in\mathbb{N}}$ is bounded in $W^{1,\infty}(\Omega)\times  \W^{1,\infty}(\Omega)$. Observe that from (\ref{c001-us}), we have that for  $h$ fixed, $\widetilde{u}^l\rightarrow \widetilde{u}$ in  $C(\overline{\Omega})$; and thus, $F'_\varepsilon(\widetilde{u}^l)\rightarrow F'_\varepsilon(\widetilde{u})$ in  $C(\overline{\Omega})$ since $F'_\varepsilon$ is a Lipschitz continuous function. Then, the linearity and continuity of $\Pi^h$ with respect to $C^0(\overline{\Omega})$-norm imply that $\Pi^h (F'_\varepsilon(\widetilde{u}^l))\rightarrow \Pi^h(F'_\varepsilon(\widetilde{u}))$ in  $C(\overline{\Omega})$ . Moreover, if we denote $(u^l,{\boldsymbol \sigma}^l)=R(\widetilde{u}^l,\widetilde{\boldsymbol \sigma}^l)$, we can deduce (recall that $\varepsilon\leq \lambda_\varepsilon(s) \leq \varepsilon^{-1}$ for all $s\in \mathbb{R}$)
\begin{eqnarray*}\label{fps02us}
&\displaystyle
\frac{1}{2k}\Vert (u^l, {\boldsymbol \sigma}^l)\Vert_{0}^{2}&\!\!\! +\frac{1}{2}\Vert {\boldsymbol \sigma}^l \Vert_1^2 \leq \displaystyle
\frac{1}{2k}\Vert (u^{n-1}_\varepsilon, {\boldsymbol \sigma}^{n-1}_\varepsilon)\Vert_{0}^{2} + C(h,k)\varepsilon^{-2} \Vert \widetilde{\boldsymbol \sigma}^l \Vert_{L^6}^2\nonumber\\
&&\!\!\!\!\!\!\!\!\!\!\!\!\!+C\varepsilon^{-2} \Vert \nabla  \Pi^h(F'_\varepsilon(\widetilde{u}^l))\Vert_{0}^2 + C(h,k)\varepsilon^{-2}\Vert \nabla  \Pi^h(F'_\varepsilon(\widetilde{u}^l))\Vert_{0}^2 \leq C,
\end{eqnarray*}
where $C$ is a constant independent of $l\in\mathbb{N}$. Therefore, $\{(u^l,{\boldsymbol \sigma}^l)=R(\widetilde{u}^l,\widetilde{\boldsymbol \sigma}^l)\}_{l\in\mathbb{N}}$ is bounded in $U_h\times {\boldsymbol \Sigma}_h\hookrightarrow W^{1,\infty}(\Omega)\times  \W^{1,\infty}(\Omega)$. Then, since we remain in finite dimension, there exists a subsequence of  $\{R(\widetilde{u}^l,\widetilde{\boldsymbol \sigma}^l)\}_{l\in\mathbb{N}}$, still denoted by$\{R(\widetilde{u}^{l},\widetilde{\boldsymbol \sigma}^{l})\}_{l\in\mathbb{N}}$,  such that  
\begin{equation}\label{c002-us}
R(\widetilde{u}^{l},\widetilde{\boldsymbol \sigma}^{l})\rightarrow (u',{\boldsymbol \sigma}') \ \ \mbox{ in } \ W^{1,\infty}(\Omega)\times  \W^{1,\infty}(\Omega).
\end{equation}
Then, from (\ref{c001-us})-(\ref{c002-us}), a standard procedure allows us to pass to the limit, as $l$ goes to $+\infty$, in (\ref{modelfexis01-usA})-(\ref{modelfexis01-usB}) (with $(\widetilde{u}^l,\widetilde{\boldsymbol \sigma}^l)$ and $(u^l,{\boldsymbol \sigma}^l)$ instead of $(\widetilde{u},\widetilde{\boldsymbol \sigma})$ and $(u,{\boldsymbol \sigma})$ respectively), and we deduce that $R(\widetilde{u},\widetilde{\boldsymbol \sigma})=({u}',{\boldsymbol \sigma}')$. Therefore, we have proved that any convergent subsequence of  $\{R(\widetilde{u}^l,\widetilde{\boldsymbol \sigma}^l)\}_{l\in\mathbb{N}}$ converges to $R(\widetilde{u},\widetilde{\boldsymbol \sigma})$ in $U_h\times {\boldsymbol \Sigma}_h$, and from uniqueness of $R(\widetilde{u},\widetilde{\boldsymbol \sigma})$, we conclude that the whole sequence $R(\widetilde{u}^l,\widetilde{\boldsymbol \sigma}^l)\rightarrow R(\widetilde{u},\widetilde{\boldsymbol \sigma})$ in $U_h\times {\boldsymbol \Sigma}_h$. Thus, $R$ is continuous.
\end{enumerate}
Therefore, the hypotheses of the Leray-Schauder fixed point theorem (in finite dimension) are satisfied and we conclude that the map $R$ has a fixed point $(u,{\boldsymbol \sigma})$, that is
$R(u,{\boldsymbol \sigma})=(u,{\boldsymbol \sigma})$, which is a solution of nonlinear scheme \textbf{US}. 
\end{proof}

\begin{lem}{\bf (Conditional uniqueness)}
If $k\, f(h,\varepsilon)<1$ (where $f(h,\varepsilon)\uparrow +\infty$ when $h\downarrow 0$ or $\varepsilon\downarrow 0$), then the solution $(u^n_\varepsilon,{\boldsymbol \sigma}_\varepsilon^n)$ of the scheme \textbf{US} is unique.
\end{lem}
\begin{proof}
Suppose that  there exist $(u^{n,1}_\varepsilon,{\boldsymbol \sigma}^{n,1}_\varepsilon),(u^{n,2}_\varepsilon,{\boldsymbol \sigma}^{n,2}_\varepsilon)\in U_h\times  {\boldsymbol \Sigma}_h$ two possible solutions of the scheme \textbf{US}. Then, defining $u=u^{n,1}_\varepsilon-u^{n,2}_\varepsilon$ and ${\boldsymbol \sigma}={\boldsymbol \sigma}^{n,1}_\varepsilon-{\boldsymbol \sigma}^{n,2}_\varepsilon$,
we have that $(u,{\boldsymbol \sigma})\in U_h\times  {\boldsymbol \Sigma}_h$ satisfies
\begin{eqnarray}\label{uniq001s}
&\displaystyle\frac{1}{k}(u,\bar{u})^h&\!\!\!+ (\lambda_\varepsilon(u^{n,1}_\varepsilon)\nabla \Pi^h (F'_\varepsilon(u^{n,1}_\varepsilon) - F'_\varepsilon(u^{n,2}_\varepsilon)),
\nabla \bar{u})+ ((\lambda_\varepsilon(u^{n,1}_\varepsilon)- \lambda_\varepsilon(u^{n,2}_\varepsilon))\nabla \Pi^h  F'_\varepsilon(u^{n,2}_\varepsilon),
\nabla \bar{u})\nonumber\\
&&\!\!\! + (\lambda_\varepsilon(u^{n,1}_\varepsilon) {\boldsymbol \sigma},
\nabla \bar{u})+ ((\lambda_\varepsilon(u^{n,1}_\varepsilon)- \lambda_\varepsilon(u^{n,2}_\varepsilon)){\boldsymbol \sigma}^{n,2}_\varepsilon,
\nabla \bar{u})=0, \ \forall
\bar{u}\in U_h,
\end{eqnarray}
\begin{eqnarray}\label{uniq002s}
&\displaystyle\frac{1}{k}({\boldsymbol \sigma},\bar{\boldsymbol \sigma})+ ( B_h {\boldsymbol \sigma},\bar{\boldsymbol \sigma})&\!\!\! = (\lambda_\varepsilon(u^{n,1}_\varepsilon)\nabla \Pi^h (F'_\varepsilon(u^{n,1}_\varepsilon) - F'_\varepsilon(u^{n,2}_\varepsilon)),\bar{\boldsymbol \sigma})\nonumber\\
&&\!\!\!+ ((\lambda_\varepsilon(u^{n,1}_\varepsilon)- \lambda_\varepsilon(u^{n,2}_\varepsilon))\nabla \Pi^h  F'_\varepsilon(u^{n,2}_\varepsilon), \bar{\boldsymbol \sigma}),\ \forall \bar{\boldsymbol \sigma}\in {\boldsymbol \Sigma}_h.
\end{eqnarray}
Taking $\bar{u}=u$, $\bar{\boldsymbol \sigma}={\boldsymbol \sigma}$ in (\ref{uniq001s})-(\ref{uniq002s}), adding the resulting expressions and using the fact that $\displaystyle\int_\Omega u=0$ as well as Remark \ref{eqh2}, estimates in Corollary \ref{welemUS} and some inverse inequalities, we obtain
\begin{eqnarray*}
&\displaystyle\frac{1}{k}&\!\!\!\!\Vert (u,{\boldsymbol \sigma})\Vert_{0}^2+ \Vert {\boldsymbol \sigma}\Vert_{1}^2 \leq \Vert \lambda_\varepsilon(u^{n,1}_\varepsilon)\Vert_{L^\infty}\Vert \nabla \Pi^h (F'_\varepsilon(u^{n,1}_\varepsilon) - F'_\varepsilon(u^{n,2}_\varepsilon))\Vert_0 \Vert \nabla u\Vert_0 \nonumber\\
&&+\Vert \lambda_\varepsilon(u^{n,1}_\varepsilon)- \lambda_\varepsilon(u^{n,2}_\varepsilon)\Vert_{L^\infty} \Vert \nabla \Pi^h  F'_\varepsilon(u^{n,2}_\varepsilon)\Vert_{0} \Vert \nabla u\Vert_0+ \Vert \lambda_\varepsilon(u^{n,1}_\varepsilon)\Vert_{L^\infty} \Vert {\boldsymbol \sigma}\Vert_0 \Vert \nabla u\Vert_0\nonumber\\
&&+ \Vert \lambda_\varepsilon(u^{n,1}_\varepsilon)- \lambda_\varepsilon(u^{n,2}_\varepsilon)\Vert_{L^\infty} \Vert {\boldsymbol \sigma}^{n,2}_\varepsilon\Vert_0 \Vert \nabla u\Vert_0  + \Vert \lambda_\varepsilon(u^{n,1}_\varepsilon)\Vert_{L^\infty}\Vert \nabla \Pi^h (F'_\varepsilon(u^{n,1}_\varepsilon) - F'_\varepsilon(u^{n,2}_\varepsilon))\Vert_0 \Vert {\boldsymbol \sigma}\Vert _0 \nonumber\\
&&+ \Vert \lambda_\varepsilon(u^{n,1}_\varepsilon)- \lambda_\varepsilon(u^{n,2}_\varepsilon)\Vert_{L^3} \Vert \nabla \Pi^h  F'_\varepsilon(u^{n,2}_\varepsilon)\Vert_0 \Vert {\boldsymbol \sigma}\Vert_{L^6}\nonumber\\
&&\!\!\!\leq \varepsilon^{-2} C(h) \Vert u\Vert_0^2 + \varepsilon^{-1} C(h) \Vert u \Vert_0^2+ \frac{1}{6} \Vert {\boldsymbol \sigma}\Vert_0 + \varepsilon^{-2} C(h) \Vert u\Vert_0^2+C_0C(h) \Vert u\Vert_0^2\nonumber\\
&& + \frac{1}{6} \Vert {\boldsymbol \sigma}\Vert_0 + \varepsilon^{-4} C(h) \Vert u\Vert_0^2 + \frac{1}{6} \Vert {\boldsymbol \sigma}\Vert_1 + \varepsilon^{-2} C(h) \Vert u\Vert_0^2  ,
\end{eqnarray*}
and therefore,
\begin{equation*}
\Vert (u,{\boldsymbol \sigma})\Vert_{0}^2+ \frac{k}{2} \Vert {\boldsymbol \sigma}\Vert_{1}^2 \leq k\, f(h,\varepsilon) \Vert u\Vert_0^2,
\end{equation*}
where $f(h,\varepsilon)\uparrow +\infty$ when $h\downarrow 0$ or $\varepsilon\downarrow 0$. Thus, if $k\, f(h,\varepsilon)< 1$, we conclude that $(u,{\boldsymbol\sigma})=(0,{\boldsymbol 0})$.
\end{proof}

\section{Scheme UZSW}
In this section, we propose an energy-stable linear fully discrete scheme associated to model (\ref{modelf00}). With this aim, we introduce the new variables 
$$z_\varepsilon= F'_\varepsilon(u_\varepsilon), \ {\boldsymbol\sigma}_\varepsilon=\nabla v_\varepsilon \ \mbox{ and } \ w_\varepsilon=\sqrt{F_\varepsilon(u_\varepsilon) + A}, \ \ \forall A>0.$$
Then, a regularized version of problem (\ref{modelf00}) in the variables $(u_\varepsilon,z_\varepsilon,{\boldsymbol \sigma}_\varepsilon,w_\varepsilon)$ is the following:
\begin{equation}
\left\{
\begin{array}
[c]{lll}%
\partial_t u_\varepsilon - \nabla \cdot(\lambda_\varepsilon (u_\varepsilon) \nabla z_\varepsilon)  -\nabla \cdot (u_\varepsilon {\boldsymbol \sigma}_\varepsilon)=0 \ \mbox{in}\ \Omega,\ t>0,\\
\partial_t {\boldsymbol \sigma}_\varepsilon + \mbox{rot(rot } {\boldsymbol \sigma}_\varepsilon \mbox{)} -\nabla(\nabla \cdot {\boldsymbol \sigma}_\varepsilon) + {\boldsymbol \sigma}_\varepsilon = u_\varepsilon\nabla z_\varepsilon\ \ \mbox{in}\ \Omega,\ t>0,\\
\displaystyle\partial_t w_\varepsilon = \frac{1}{2\sqrt{F_\varepsilon(u_\varepsilon)+A}}F'_\varepsilon(u_\varepsilon)\, \partial_t u_\varepsilon \ \ \mbox{in}\ \Omega,\ t>0,\vspace{0.2 cm}\\
z_\varepsilon=\displaystyle\frac{1}{\sqrt{F_\varepsilon(u_\varepsilon)+A}}F'_\varepsilon(u_\varepsilon)\, w_\varepsilon  \ \ \mbox{in}\ \Omega,\ t>0,\\
\displaystyle\frac{\partial z_\varepsilon}{\partial \mathbf{n}}=0\quad \mbox{on}\ \partial\Omega,\ t>0,\\
{\boldsymbol \sigma}_\varepsilon\cdot \mathbf{n}=0, \ \ \left[\mbox{rot }{\boldsymbol \sigma}_\varepsilon \times \mathbf{n}\right]_{tang}=0 \quad \mbox{on}\ \partial\Omega,\ t>0,\\
u_\varepsilon(\x ,0)=u_0(\x )\geq 0,\ {\boldsymbol \sigma}_\varepsilon(\x ,0)=\nabla v_0(
\x ),\ w_\varepsilon(\x,0)=\sqrt{F_\varepsilon(u_0(\x))+A}\quad \mbox{in}\ \Omega,
\end{array}
\right.  \label{modelf02acontUZSW}
\end{equation}
for all constant $A>0$. 
\begin{obs}
Notice that problems (\ref{modelf02acontUS}) and (\ref{modelf02acontUZSW}) are equivalents for all $A>0$. In fact, if $(u_\varepsilon,{\boldsymbol \sigma}_\varepsilon)$ is a solution of the scheme \textbf{US}, then defining $z_\varepsilon= F'_\varepsilon(u_\varepsilon)$ and $w_\varepsilon=\sqrt{F_\varepsilon(u_\varepsilon) + A}$, we deduce that $(u_\varepsilon,z_\varepsilon,{\boldsymbol \sigma}_\varepsilon,w_\varepsilon)$ is a solution of the scheme \textbf{UZSW}. Reciprocally, if  $(u_\varepsilon,z_\varepsilon,{\boldsymbol \sigma}_\varepsilon,w_\varepsilon)$ is a solution of the scheme \textbf{UZSW}, then from
$$
w_\varepsilon=\sqrt{F_\varepsilon(u_\varepsilon) + A} \ \Longleftrightarrow \ 
\left\{\begin{array}{lcl}
\displaystyle\partial_t w_\varepsilon = \frac{1}{2\sqrt{F_\varepsilon(u_\varepsilon)+A}}F'_\varepsilon(u_\varepsilon)\, \partial_t u_\varepsilon,\\
w_\varepsilon\vert_{t=0}=\sqrt{F_\varepsilon(u_0)+A},
\end{array}\right. 
$$
and (\ref{modelf02acontUZSW})$_4$, we deduce that $z_\varepsilon= F'_\varepsilon(u_\varepsilon)$, and therefore, $(u_\varepsilon,{\boldsymbol \sigma}_\varepsilon)$ is a solution of the scheme \textbf{US}.
\end{obs}
  As in the previous section, once solved (\ref{modelf02acontUZSW}), we can recover $v_\varepsilon$ from $u_\varepsilon$ solving (\ref{modelf01eqv}). Observe that multiplying (\ref{modelf02acontUZSW})$_1$ by $z_\varepsilon$, (\ref{modelf02acontUZSW})$_2$ by ${\boldsymbol\sigma}_\varepsilon$, (\ref{modelf02acontUZSW})$_3$ by $2 w_\varepsilon$, (\ref{modelf02acontUZSW})$_4$ by $\partial_t u_\varepsilon$, integrating over $\Omega$ and using the boundary conditions of (\ref{modelf02acontUZSW}), we obtain the following energy law
\begin{eqnarray*}
\frac{d}{dt} \displaystyle \int_\Omega \Big( \vert w_\varepsilon\vert^2 + \frac{1}{2} \vert {\boldsymbol\sigma}_\varepsilon\vert^2\Big) d \x + \int_\Omega \lambda_\varepsilon(u_\varepsilon)\vert \nabla z_\varepsilon\vert^2  d \x + \Vert{\boldsymbol\sigma}_\varepsilon \Vert_1^2= 0.
\end{eqnarray*}
In particular, the modified energy $\mathcal{E}(w,{\boldsymbol\sigma})= \displaystyle\int_\Omega \Big(\vert w\vert^2 + \frac{1}{2} \vert {\boldsymbol\sigma}\vert^2\Big) d \x$ is decreasing in time. Then, we consider a fully discrete approximation of the regularized problem (\ref{modelf02acontUZSW}) using a FE discretization in space and a first order semi-implicit discretization in time (again considered for simplicity on a uniform partition of $[0,T]$ with time step
$k=T/N : (t_n = nk)_{n=0}^{n=N}$). Concerning the space discretization, we consider the triangulation as in the scheme \textbf{US} (hence, the constraint ({\bf H}) related with the right-angles simplices is not imposed), and we choose the following continuous FE spaces for  $u_\varepsilon$, $z_\varepsilon$, ${\boldsymbol\sigma}_\varepsilon$, $w_\varepsilon$ and $v_\varepsilon$:
$$(U_h, Z_h,{\boldsymbol\Sigma}_h,W_h, V_h) \subset H^1(\Omega)^5,\quad \hbox{generated by $\mathbb{P}_k,\mathbb{P}_l,\mathbb{P}_{m},\mathbb{P}_r,\mathbb{P}_{s}$ with $k,l,m,r,s\geq 1$ and $k\leq l$.}
$$
\begin{obs}
The constraint $k\leq l$ implies $U_h\subseteq Z_h$ which will be used to prove the well-posedness of the scheme \textbf{UZSW} (see Theorem \ref{WPuzsw} below).
\end{obs}
Then, we consider the following first order in time, linear and coupled scheme: 
\begin{itemize}
\item{\underline{\emph{Scheme \textbf{UZSW}}:}\\
{\bf Initialization}: Let $(u_h^{0},{\boldsymbol \sigma}_h^{0},w^0_{h,\varepsilon})=(Q^h u_0, \widetilde{Q}^h (\nabla v_0),  \widehat{Q}^h (\sqrt{F_\varepsilon(u_0)+A})) \in  U_h\times {\boldsymbol\Sigma}_h\times W_h$.\\
{\bf Time step} n: Given $(u^{n-1}_\varepsilon,{\boldsymbol \sigma}^{n-1}_\varepsilon,w^{n-1}_\varepsilon)\in  U_h\times {\boldsymbol\Sigma}_h\times W_h$, compute $(u^{n}_\varepsilon,z^n_\varepsilon,{\boldsymbol \sigma}^{n}_\varepsilon,w^n_\varepsilon)\in  U_h\times Z_h \times {\boldsymbol\Sigma}_h\times W_h$ solving
\begin{equation}
\left\{
\begin{array}
[c]{lll}%
(\delta_t u^n_\varepsilon,\bar{z}) + (\lambda_\varepsilon (u^{n-1}_\varepsilon) \nabla z^n_\varepsilon,\nabla \bar{z}) = -(u^{n-1}_\varepsilon {\boldsymbol \sigma}^n_\varepsilon,\nabla \bar{z}), \ \ \forall \bar{z}\in Z_h,\\
(\delta_t {\boldsymbol \sigma}^n_\varepsilon,\bar{\boldsymbol \sigma}) + ( B_h{\boldsymbol \sigma}^n_\varepsilon,\bar{\boldsymbol \sigma}) =
(u^{n-1}_\varepsilon \nabla z^n_\varepsilon,\bar{\boldsymbol \sigma}),\ \ \forall
\bar{\boldsymbol \sigma}\in \Sigma_h,\\
(\delta_t w^n_\varepsilon,\bar{w}) = \Big(\frac{1}{2\sqrt{F_\varepsilon(u^{n-1}_\varepsilon)+A}}F'_\varepsilon(u^{n-1}_\varepsilon)\,\delta_t u^n_\varepsilon,\bar{w}\Big), \ \ \forall \bar{w}\in W_h,\\ 
(z^n_\varepsilon,\bar{u}) = \Big(\frac{1}{\sqrt{F_\varepsilon(u^{n-1}_\varepsilon)+A}}F'_\varepsilon(u^{n-1}_\varepsilon)\, w^n_\varepsilon,\bar{u}\Big), \ \ \forall \bar{u}\in U_h.
\end{array}
\right.  \label{modelf02c}
\end{equation}}
\end{itemize}
Recall that $(B_h{\boldsymbol \sigma}^n_\varepsilon,\bar{\boldsymbol \sigma}) := (\mbox{rot } {\boldsymbol\sigma}_\varepsilon^n,\mbox{rot } \bar{\boldsymbol\sigma}) + (\nabla \cdot {\boldsymbol\sigma}_\varepsilon^n, \nabla \cdot \bar{\boldsymbol\sigma}) + ({\boldsymbol\sigma}_\varepsilon^n, \bar{\boldsymbol\sigma})$ for all
$\bar{\boldsymbol \sigma}\in \Sigma_h$, $Q^h$ is the $L^2$-projection on $U_h$ defined in (\ref{MLP2}), and $\widetilde{Q}^h$ and $\widehat{Q}^h$ are the standard $L^2$-projections on ${\boldsymbol\Sigma}_h$ and $W_h$ respectively. As in the scheme \textbf{US}, once the scheme \textbf{UZSW} is solved, given $v^{n-1}_\varepsilon\in V_h$, we can recover $v^n_\varepsilon=v^n_\varepsilon(u^n_\varepsilon) \in V_h$ solving (\ref{edovf}).

\subsection{Mass conservation and Energy-stability}
Observe that the scheme \textbf{UZSW}  is also conservative in $u$ (satisfying \eqref{conu1}) and also has the  behavior for $\int_\Omega v_n$ given in \eqref{conv1}.

\begin{defi}
A numerical scheme with solution $(u^{n}_\varepsilon,z^n_\varepsilon,{\boldsymbol \sigma}^{n}_\varepsilon,w^n_\varepsilon)$ is called energy-stable with respect to the  energy 
\begin{equation}\label{Euzsw}
\mathcal{E}(w,{\boldsymbol\sigma})= \Vert w\Vert_0^2 + \frac{1}{2} \Vert {\boldsymbol\sigma}\Vert_0^2
\end{equation}
if this energy is time decreasing, that is, $\mathcal{E}(w^n_\varepsilon,{\boldsymbol\sigma}^n_\varepsilon)\leq \mathcal{E}(w^{n-1}_\varepsilon,{\boldsymbol\sigma}^{n-1}_\varepsilon)$ for all $n\geq 1$.
\end{defi}

\begin{tma} {\bf (Unconditional stability)} \label{estinc1uzsw}
The scheme \textbf{UZSW} is unconditional energy stable with respect to $\mathcal{E}(w,{\boldsymbol\sigma})$. In fact, if $(u^{n}_\varepsilon,z^n_\varepsilon,{\boldsymbol \sigma}^{n}_\varepsilon,w^n_\varepsilon)$ is a solution of \textbf{UZSW}, then the following discrete energy law holds
\begin{eqnarray}\label{lawenerfydisceC}
\delta_t \mathcal{E}(w^n_\varepsilon,{\boldsymbol \sigma}^n_\varepsilon)+ 
k
\Vert\delta_t w^n_\varepsilon\Vert_0^2+
\frac{k}{2} \Vert \delta_t {\boldsymbol \sigma}^n_\varepsilon\Vert_{0}^2 
+\int_\Omega \lambda_\varepsilon(u^{n-1}_\varepsilon) \vert
\nabla z^n_\varepsilon\vert^{2} +
\Vert  {\boldsymbol
\sigma}^n_\varepsilon\Vert_{1}^{2}=0.
\end{eqnarray}
\end{tma}
\begin{proof}
The proof follows taking $(\bar{z},\bar{\boldsymbol \sigma},\bar{w},\bar{u})=(z^n_\varepsilon,{\boldsymbol
\sigma}^n_\varepsilon,2w^n_\varepsilon,\delta_t u^n_\varepsilon)$ in (\ref{modelf02c}).
\end{proof}

From the (local in time) discrete energy law (\ref{lawenerfydisceC}), we deduce the following global in time estimates for $(u^n_\varepsilon,z^n_\varepsilon,{\boldsymbol
\sigma}^n_\varepsilon,w^n_\varepsilon)$ solution of the scheme \textbf{UZSW}:

\begin{cor} {\bf(Uniform Weak estimates) }
Assume that $(u_0,v_0)\in L^2(\Omega)\times H^1(\Omega)$. Let $(u^n_\varepsilon,z^n_\varepsilon,{\boldsymbol
\sigma}^n_\varepsilon,w^n_\varepsilon)$ be a solution of scheme \textbf{UZSW}. Then, the following estimate holds
\begin{equation}\label{weeUZSW}
\Vert w^n_\varepsilon\Vert_0^2 + \Vert {\boldsymbol
\sigma}^n_\varepsilon\Vert_{0}^{2}+
k \underset{m=1}{\overset{n}{\sum}}
\left( \int_\Omega \lambda_\varepsilon(u^{m-1}_\varepsilon) \vert
\nabla z^m_\varepsilon\vert^{2} +
\Vert  {\boldsymbol
\sigma}^m_\varepsilon\Vert_{1}^{2}\right)\leq C_0, \ \ \ \forall n\geq 1,
\end{equation}
with the constant $C_0>0$ depending on the data $(\Omega, u_0,v_0)$, but independent of $k,h,n$ and $\varepsilon$.
\end{cor}
\begin{proof}
Proceeding as in (\ref{ee5}) and taking into account that $u_0\geq 0$ and
$(u_h^{0},{\boldsymbol \sigma}_h^{0},w^0_{h,\varepsilon})=$ \\ $(Q^h u_0, \widetilde{Q}^h (\nabla v_0),  \widehat{Q}^h (\sqrt{F_\varepsilon(u_0(\x))+A}))$,  we have that
\begin{eqnarray}\label{ee7}
&\Vert w_{h,\varepsilon}^0\Vert_0^2+ \displaystyle\frac{1}{2}\Vert {\boldsymbol
\sigma}^0_h\Vert_{0}^{2}&\!\!\! = \Vert  \widehat{Q}^h (\sqrt{F_\varepsilon(u_0)+A})\Vert_0^2+ \displaystyle\frac{1}{2}\Vert \widetilde{Q}^h (\nabla v_0)\Vert_{0}^{2} \leq \int_\Omega (F_\varepsilon(u_0) +A) +  \frac{1}{2}\Vert \nabla v_0\Vert_{0}^{2}\nonumber\\
&&\!\!\! \leq C\int_\Omega ((u_0)^2 + 1) +  \frac{1}{2}\Vert \nabla v_0\Vert_{0}^{2} \leq C(\Vert u_0\Vert_0^2 + \Vert v_0\Vert_{1}^{2} + 1)\leq C_0,
\end{eqnarray}
 with the constant $C_0>0$ depending on the data $(\Omega,u_0,v_0)$, but independent of $k,h,n$ and $\varepsilon$. Therefore, multiplying the discrete energy law (\ref{lawenerfydisceC}) by $k$, adding from $m=1$ to $m=n$ and using (\ref{ee7}),  we arrive at (\ref{weeUZSW}). 
\end{proof}

\subsection{Well-posedness}
\begin{tma} {\bf(Unconditional unique solvability)}\label{WPuzsw}
There exists a unique $(u^n_\varepsilon,z^n_\varepsilon,{\boldsymbol
\sigma}^n_\varepsilon,w^n_\varepsilon)$ solution of scheme \textbf{UZSW}.
\end{tma}
\begin{proof}
By linearity of the scheme \textbf{UZSW}, it suffices to prove uniqueness. Suppose that  there exist $(u^n_{\varepsilon,1},z^n_{\varepsilon,1},{\boldsymbol
\sigma}^n_{\varepsilon,1},w^n_{\varepsilon,1}),(u^n_{\varepsilon,2},z^n_{\varepsilon,2},{\boldsymbol
\sigma}^n_{\varepsilon,2},w^n_{\varepsilon,2})\in U_h\times Z_h\times {\boldsymbol \Sigma}_h\times W_h$ two possible solutions of \textbf{UZSW}.
Then defining $u=u^n_{\varepsilon,1}-u^n_{\varepsilon,2}$, $z=z^n_{\varepsilon,1} - z^n_{\varepsilon,2}$, ${\boldsymbol
\sigma}={\boldsymbol
\sigma}^n_{\varepsilon,1}-{\boldsymbol
\sigma}^n_{\varepsilon,2}$ and $w=w^n_{\varepsilon,1} - w^n_{\varepsilon,2}$,
we have that $(u,z,{\boldsymbol
\sigma},w)\in U_h\times Z_h\times {\boldsymbol \Sigma}_h\times W_h$ satisfies
\begin{equation}
\left\{
\begin{array}
[c]{lll}%
\frac{1}{k}(u,\bar{z}) + (\lambda_\varepsilon (u^{n-1}_\varepsilon) \nabla z,\nabla \bar{z}) = -(u^{n-1}_\varepsilon {\boldsymbol \sigma},\nabla \bar{z}), \ \ \forall \bar{z}\in Z_h,\\
\frac{1}{k}({\boldsymbol \sigma},\bar{\boldsymbol \sigma}) + ( B_h{\boldsymbol \sigma},\bar{\boldsymbol \sigma}) =
(u^{n-1}_\varepsilon \nabla z,\bar{\boldsymbol \sigma}),\ \ \forall
\bar{\boldsymbol \sigma}\in \Sigma_h,\\
\frac{1}{k}(w,\bar{w}) = \frac{1}{2k}\Big(\frac{1}{\sqrt{F_\varepsilon(u^{n-1}_\varepsilon)+A}}F'_\varepsilon(u^{n-1}_\varepsilon)\,u,\bar{w}\Big), \ \ \forall \bar{w}\in W_h,\\ 
(z,\bar{u}) = \Big(\frac{1}{\sqrt{F_\varepsilon(u^{n-1}_\varepsilon)+A}}F'_\varepsilon(u^{n-1}_\varepsilon)\,w ,\bar{u}\Big), \ \ \forall \bar{u}\in U_h.
\end{array}
\right.  \label{modelf02clin}
\end{equation}
Taking $(\bar{z},\bar{\boldsymbol \sigma},\bar{w},\bar{u})=(z,{\boldsymbol
\sigma},2w,\frac{1}{k} u)$ in (\ref{modelf02clin}) and adding, we obtain
\begin{equation*}
\frac{2}{k}\Vert w\Vert_0^2+ \frac{1}{k}\Vert {\boldsymbol
\sigma}\Vert_0^{2}
+\int_\Omega \lambda_\varepsilon(u^{n-1}_\varepsilon) \vert
\nabla z\vert^{2} +
\Vert  {\boldsymbol
\sigma}\Vert_{1}^{2}=0.
\end{equation*}
Taking into account that $\lambda_\varepsilon(u^{n-1}_\varepsilon)\geq \varepsilon$, we deduce that $(\nabla z,{\boldsymbol
\sigma},w)=({\boldsymbol 0},{\boldsymbol 0},0)$, hence $z=C:=cte$. Moreover, using the fact that $w=0$ and $z=C$, from (\ref{modelf02clin})$_4$ we conclude that $z=0$. Finally, taking $\bar{z}=u$ in (\ref{modelf02clin})$_1$ (which is possible thanks to the choice $U_h\subseteq Z_h$), since $(\nabla z,{\boldsymbol
\sigma})=({\boldsymbol 0},{\boldsymbol 0})$ we conclude $u=0$.
\end{proof}

\section{Numerical simulations}
The aim of this section is to compare the results of several numerical simulations using the schemes derived throughout the paper. We choose the spaces for $(u,z,{\boldsymbol\sigma},w)$  generated by $\mathbb{P}_1$-continuous FE. Moreover, we have chosen the 2D domain $[0,2]^2$ with a structured mesh (then ({\bf H}) holds and the scheme \textbf{UV} can be defined), and all the simulations are carried out using \textbf{FreeFem++} software. In the comparison, we will also consider the classical Backward Euler scheme for model (\ref{modelf00}), which is given for the following first order in time, nonlinear and coupled scheme: 
\begin{itemize}
\item{\underline{\emph{Scheme \textbf{BEUV}}:}\\
{\bf Initialization}: Let $(u^{0},{v}^{0})=(Q^h u_0,R^h {v}_{0})\in  U_h\times V_h$. \\
{\bf Time step} n: Given $(u^{n-1},{v}^{n-1})\in  U_h\times {V}_h$, compute $(u^n,{v}^{n})\in U_h \times {V}_h$ solving
\begin{equation*}
\left\{
\begin{array}
[c]{lll}%
(\delta_t u^n,\bar{u}) + (\nabla u^n,\nabla \bar{u}) = -(u^{n}\nabla {v}^n,\nabla \bar{u}), \ \ \forall \bar{u}\in U_h,\\
(\delta_t {v}^n,\bar{v}) +(A_h  v^n, \bar{v})  =
(u^{n} ,\bar{v}),\ \ \forall
\bar{v}\in V_h.
\end{array}
\right.
\end{equation*}}
\end{itemize}
\begin{obs}\label{RBE}
The scheme \textbf{BEUV} has not been analyzed in the previous sections because it is not clear how to prove its energy-stability. In fact, observe that the scheme \textbf{UV} (which is the ``closest'' approximation to the scheme \textbf{BEUV} considered in this paper) differs from the scheme \textbf{BEUV} in the use of the regularized functions $F_\varepsilon$, $F'_\varepsilon$ and $F''_\varepsilon$ (see (\ref{F2pE}) and Figure \ref{fig:Fe}) and in the approximation of cross-diffusion term $(u\nabla {v},\nabla \bar{u})$, which are crucial for the proof of the energy-stability of the scheme \textbf{UV}. 
\end{obs}
The linear iterative methods used to approach the solutions of the nonlinear schemes \textbf{UV}, \textbf{US} and \textbf{BEUV} are the following Picard methods, in which, we denote $(u^{n}_\varepsilon,{v}_\varepsilon^{n},{\boldsymbol \sigma}_\varepsilon^{n}):=(u^{n},{v}^{n},{\boldsymbol \sigma}^{n})$.
\begin{enumerate}
\item[(i)]{Picard method to approach a solution $(u^{n},{v}^{n})$ of the scheme \textbf{UV}}\\
{\bf Initialization ($l=0$):} Set $(u^{0},{v}^{0})=(u^{n-1},{v}^{n-1})\in  U_h\times V_h$.\\
{\bf Algorithm:} Given $(u^{l},{v}^{l})\in  U_h\times V_h$, compute $(u^{l+1},{v}^{l+1})\in  U_h\times V_h$ such that
$$
\left\{
\begin{array}
[c]{lll}%
\frac{1}{k}(u^{l+1},\bar{u})^h + (\nabla u^{l+1},\nabla \bar{u}) = \frac{1}{k}(u^{n-1},\bar{u})^h -(\Lambda_\varepsilon (u^{l})\nabla {v}^{l+1},\nabla \bar{u}), \ \ \forall \bar{u}\in U_h,\\
\frac{1}{k}({v}^{l+1} ,\bar{v}) +(A_h  v^{l+1}, \bar{v})  = \frac{1}{k}({v}^{n-1},\bar{v}) 
+(u^{l},\bar{v}),\ \ \forall
\bar{v}\in V_h.
\end{array}
\right.  
$$
until the stopping criteria $\max\left\{\displaystyle\frac{\Vert
u^{l+1} - u^{l}\Vert_{0}}{\Vert u^{l}\Vert_{0}},\displaystyle\frac{\Vert
v^{l+1} - v^{l}\Vert_{0}}{\Vert
v^{l}\Vert_{0}}\right\}\leq tol$. 

\item[(ii)]{Picard method to approach a solution $(u^{n},{\boldsymbol \sigma}^{n})$ of the scheme \textbf{US}}\\
{\bf Initialization ($l=0$):} Set $(u^{0},{\boldsymbol \sigma}^{0})=(u^{n-1},{\boldsymbol \sigma}^{n-1})\in  U_h\times {\boldsymbol \Sigma}_h$.\\
{\bf Algorithm:} Given $(u^{l},{\boldsymbol \sigma}^{l})\in  U_h\times {\boldsymbol \Sigma}_h$, compute $(u^{l+1},{\boldsymbol \sigma}^{l+1})\in  U_h\times {\boldsymbol \Sigma}_h$ such that
\begin{equation*}
\left\{
\begin{array}
[c]{lll}%
\frac{1}{k}(u^{l+1},\bar{u})^h + (\nabla (u^{l+1}, \nabla \bar{u}) -  (\nabla u^{l}, \nabla \bar{u}) \\
\hspace{2 cm} = \frac{1}{k}(u^{n-1},\bar{u})^h - (\lambda_\varepsilon (u^{l}) \nabla \Pi^h(F'_\varepsilon(u^{l})),\nabla \bar{u})  -(\lambda_\varepsilon (u^{l}) {\boldsymbol\sigma}^{l+1},\nabla \bar{u}), \ \ \forall \bar{u}\in U_h,\\
\frac{1}{k}({\boldsymbol \sigma}^{l+1},\bar{\boldsymbol \sigma}) + \langle B{\boldsymbol \sigma}^{l+1},\bar{\boldsymbol \sigma}\rangle =\frac{1}{k}({\boldsymbol \sigma}^{n-1},\bar{\boldsymbol \sigma}) +
(\lambda_\varepsilon (u^{l}) \nabla  \Pi^h(F'_\varepsilon(u^{l})),\bar{\boldsymbol \sigma}),\ \ \forall
\bar{\boldsymbol \sigma}\in \Sigma_h.
\end{array}
\right.  
\end{equation*}
until the stopping criteria $\max\left\{\displaystyle\frac{\Vert
u^{l+1} - u^{l}\Vert_{0}}{\Vert u^{l}\Vert_{0}},\displaystyle\frac{\Vert
{\boldsymbol \sigma}^{l+1} - {\boldsymbol \sigma}^{l}\Vert_{0}}{\Vert
{\boldsymbol \sigma}^{l}\Vert_{0}}\right\}\leq tol$. Note that a residual term $(\nabla (u^{l+1} - u^{l}), \nabla \bar{u})$ is considered.

\item[(iii)]{Picard method to approach a solution $(u^{n},{v}^{n})$ of  the scheme \textbf{BEUV}}\\
{\bf Initialization ($l=0$):} Set $(u^{0},{v}^{0})=(u^{n-1},{v}^{n-1})\in  U_h\times V_h$.\\
{\bf Algorithm:} Given $(u^{l},{v}^{l})\in  U_h\times V_h$, compute $(u^{l+1},{v}^{l+1})\in  U_h\times V_h$ such that
\begin{equation*}
\left\{
\begin{array}
[c]{lll}%
\frac{1}{k}(u^{l+1} ,\bar{u}) + (\nabla u^{l+1},\nabla \bar{u}) = \frac{1}{k}( u^{n-1},\bar{u})  -(u^{l}\nabla {v}^{l+1},\nabla \bar{u}), \ \ \forall \bar{u}\in U_h,\\
\frac{1}{k}({v}^{l+1},\bar{v}) +(A_h  v^{l+1}, \bar{v})  = \frac{1}{k}({v}^{n-1},\bar{v})+
(u^{l} ,\bar{v}),\ \ \forall
\bar{v}\in V_h,
\end{array}
\right. 
\end{equation*}
until the stopping criteria $\max\left\{\displaystyle\frac{\Vert
u^{l+1} - u^{l}\Vert_{0}}{\Vert u^{l}\Vert_{0}},\displaystyle\frac{\Vert
v^{l+1} - v^{l}\Vert_{0}}{\Vert
v^{l}\Vert_{0}}\right\}\leq tol$.

\end{enumerate}
\begin{obs}
In all cases, first we compute $v^{l+1}$ (resp. ${\boldsymbol\sigma}^{l+1}$) solving the $v$-equation (resp. ${\boldsymbol\sigma}$-system) and then, inserting $v^{l+1}$ (resp. ${\boldsymbol\sigma}^{l+1}$)  in $u$-equation, we compute $u^{l+1}$. 
\end{obs}

\subsection{Positivity of $u^n$}
In this subsection, we compare the positivity of the variable $u^n\in U_h$ in the four schemes. Here, we choose the space $V_h$  generated by $\mathbb{P}_2$-continuous FE. We recall that for the three schemes studied in this paper, namely schemes \textbf{UV}, \textbf{UZSW} and \textbf{US}, it is not clear the positivity of the variable $u^n$. Moreover, for the schemes \textbf{UV} and \textbf{US}, it was proved that $\Pi^h(u^n_{\varepsilon-})\rightarrow 0$ in $L^2(\Omega)$ as $\varepsilon\rightarrow 0$ (see Remarks \ref{NNuh} and \ref{NNuhs}); while for the scheme \textbf{UZSW} this fact is not clear. For this reason, in Figs. \ref{fig:MUuv}-\ref{fig:MUus} we compare the positivity of the variable $u^n_\varepsilon$ in the schemes, taking $\varepsilon=10^{-3}$, $\varepsilon=10^{-5}$ and $\varepsilon=10^{-8}$. In the scheme \textbf{UZSW} we fix $A=1$ (and thus, $F_\varepsilon(s)+A\geq 1$ for all $s\in \mathbb{R}$). We consider the time step $k=10^{-5}$, the tolerance parameter for the linear iterative methods $tol=10^{-4}$ and the initial conditions (see Fig.~\ref{fig:initcond1})
$$u_0\!\!=\!\!-10xy(2-x)(2-y)exp(-10(y-1)^2-10(x-1)^2)+10.0001,$$
$$v_0\!\!=\!\!100xy(2-x)(2-y)exp(-30(y-1)^2-30(x-1)^2)+0.0001.$$
\begin{figure}[htbp]
\centering 
\subfigure[Initial cell density $u_0$]{\includegraphics[width=65mm]{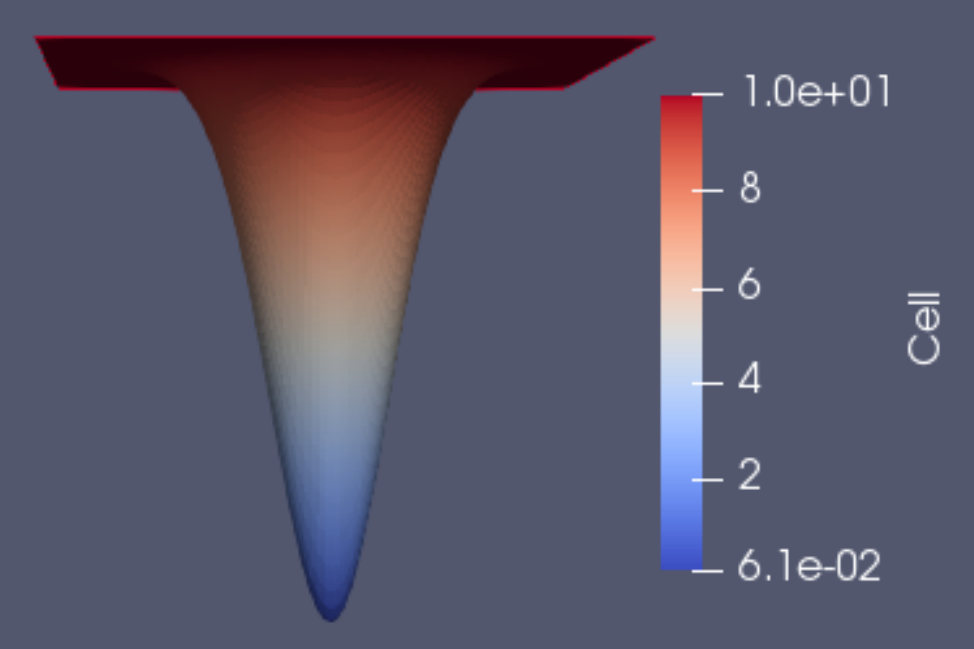}} \hspace{1,2 cm} 
\subfigure[Initial chemical concentration $v_0$]{\includegraphics[width=65mm]{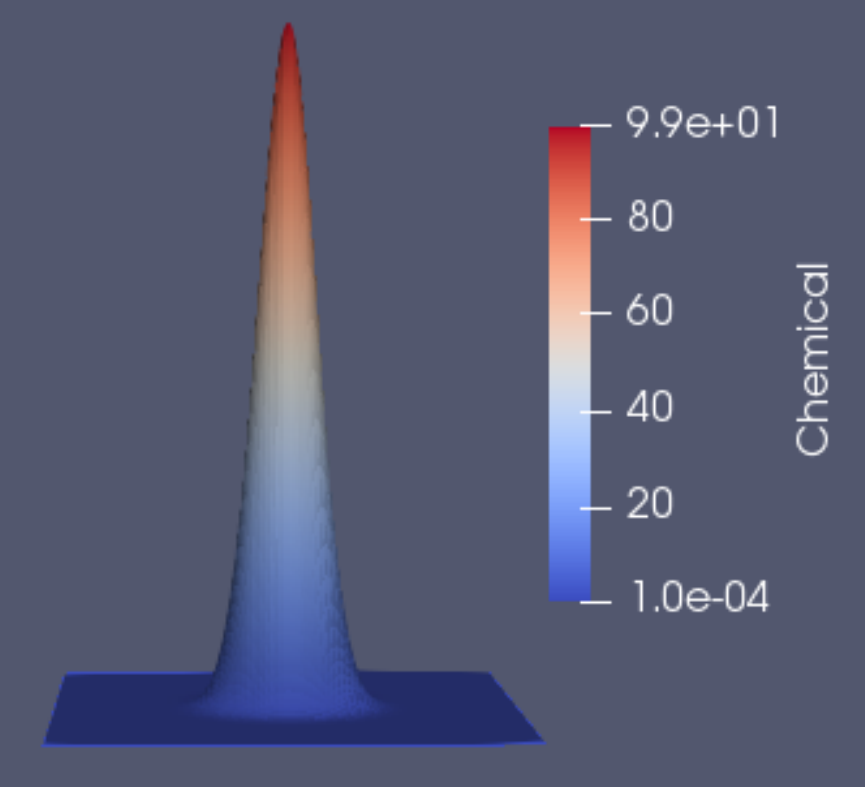}} 
\caption{Initial conditions.} \label{fig:initcond1}
\end{figure}
Note that $u_0,v_0>0$ in $\Omega$, $\min (u_0)=u_0(1,1)=0.0001$ and $\max (v_0)=v_0(1,1)=100.0001$. Moreover, for the schemes $\textbf{UV}$ and $\textbf{UZSW}$ we take the mesh size $h=\frac{1}{40}$, while for the scheme $\textbf{US}$ it was necessary to take $h=\frac{1}{80}$, because for thicker meshes we had convergence problems of the iterative method. \\
In the case of the schemes $\textbf{UV}$ and $\textbf{US}$, we observe that although $u^n_\varepsilon$ is ne\-ga\-tive for some $\x \in \Omega$ in some times $t_n>0$, when $\varepsilon\rightarrow 0$ these values are closer to $0$; while in the case of  the scheme $\textbf{UZSW}$, this same behavior is not observed (see Figs.~\ref{fig:MUuv}-\ref{fig:MUus}). Finally, in the case of the scheme $\textbf{BEUV}$ (see Fig. \ref{fig:MUbeuv}), we have also observed negative values for the minimum of $u^n$ in some times $t_n>0$, with more negative values than in the schemes $\textbf{UV}$ and $\textbf{US}$. 
 \begin{figure}[h]
  \begin{center}
 {\includegraphics[height=0.5\linewidth]{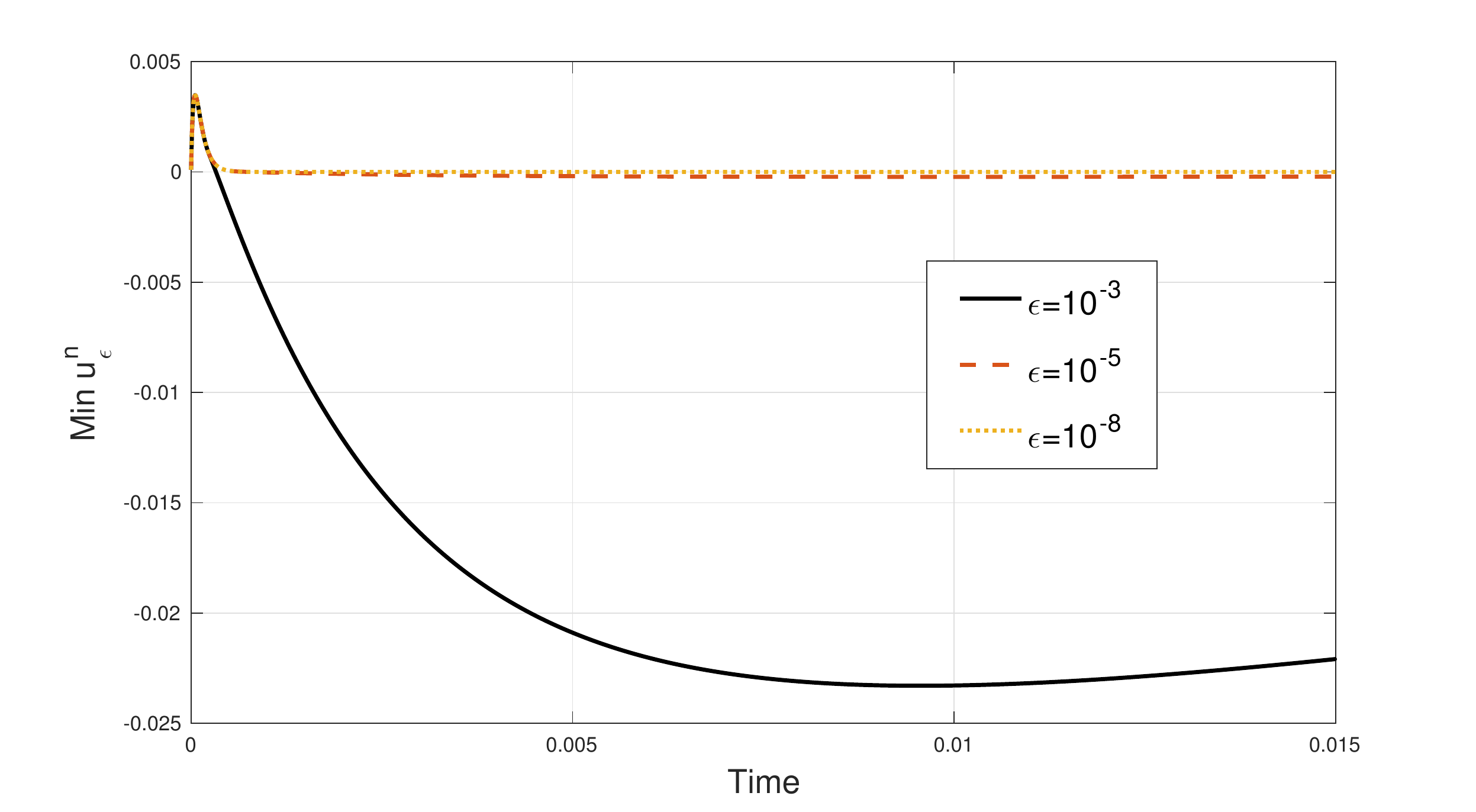}}
 \caption{Minimum values of $u^n_\varepsilon$ computed using the scheme \textbf{UV}.
  \label{fig:MUuv}}
  \end{center}
\end{figure}
 \begin{figure}[h]
  \begin{center}
  {\includegraphics[height=0.5\linewidth]{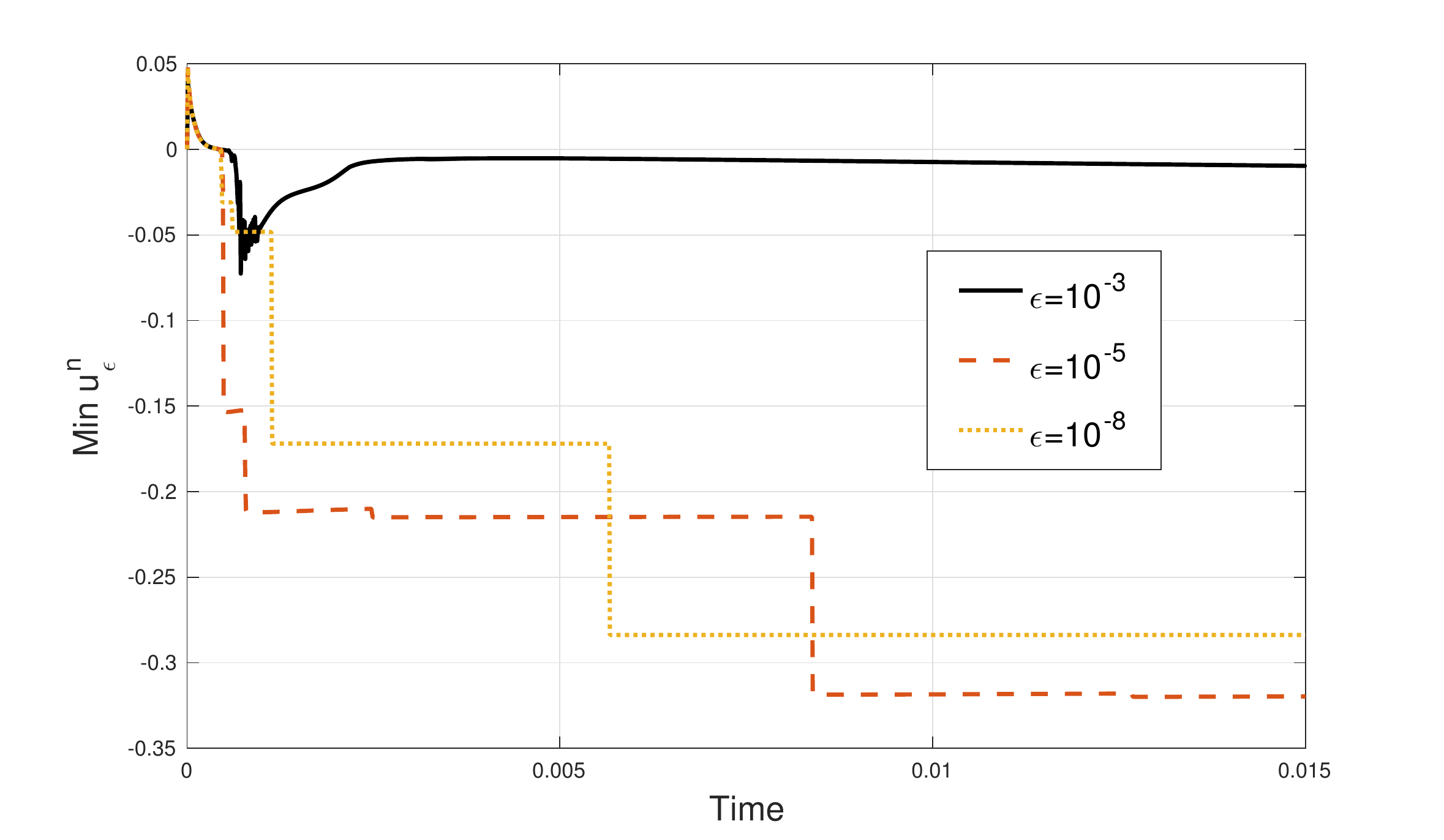}}
  \caption{Minimum values of $u^n_\varepsilon$ computed using the scheme \textbf{UZSW}.
  \label{fig:MUuzsw}}
  \end{center}
\end{figure}
 \begin{figure}[h]
  \begin{center}
  {\includegraphics[height=0.5\linewidth]{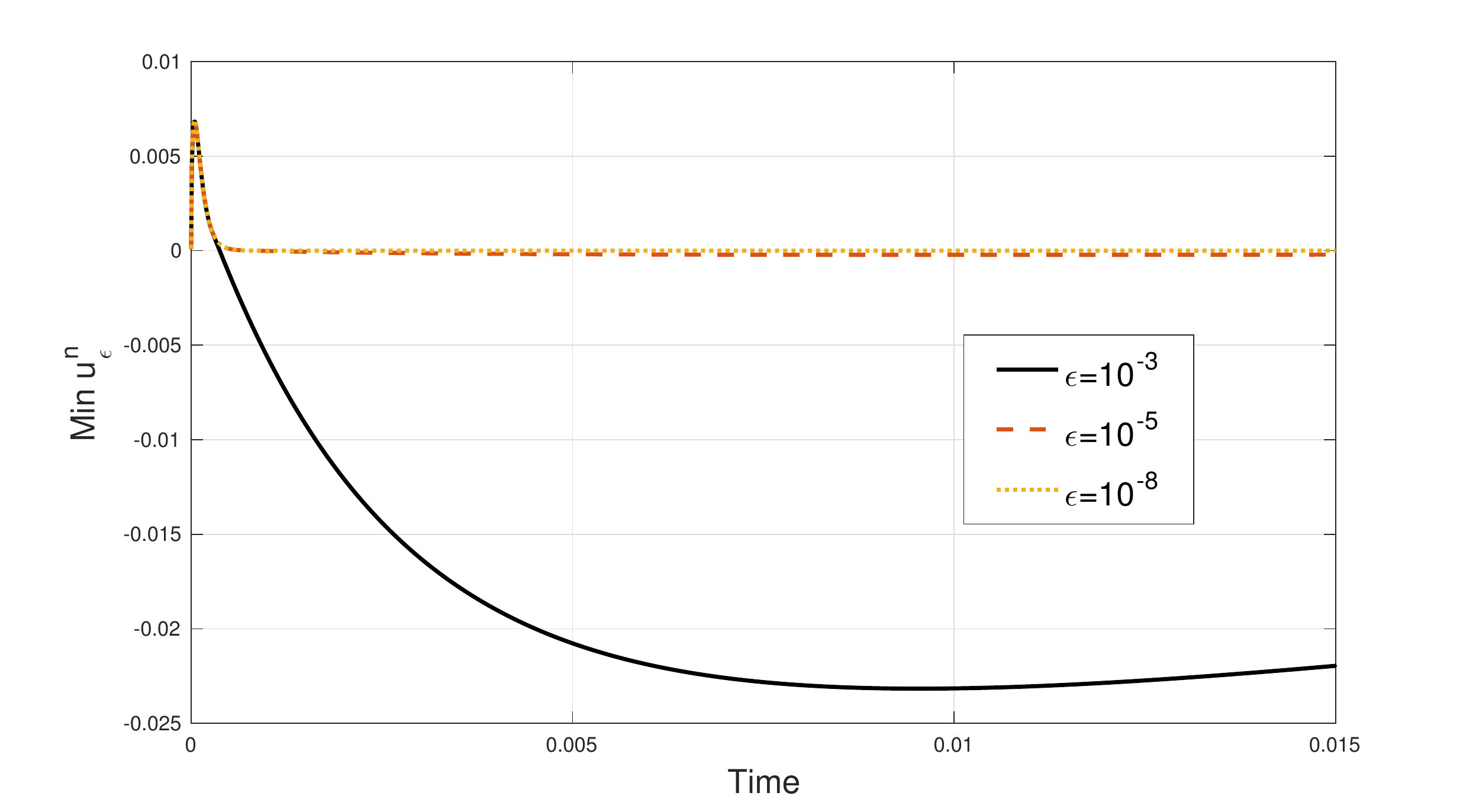}}
  \caption{Minimum values of $u^n_\varepsilon$ computed using the scheme \textbf{US}.
  \label{fig:MUus}}
  \end{center}
\end{figure}
\begin{obs}
In Figs.~\ref{fig:MUuv} and \ref{fig:MUus} there are also negative values of minimum of $u^n_\varepsilon$ for $\varepsilon=10^{-5}$ and $\varepsilon=10^{-8}$, but those are of order $10^{-4}$ and $10^{-7}$ respectively in both figures.
\end{obs}
 \begin{figure}[h]
  \begin{center}
  {\includegraphics[height=0.5\linewidth]{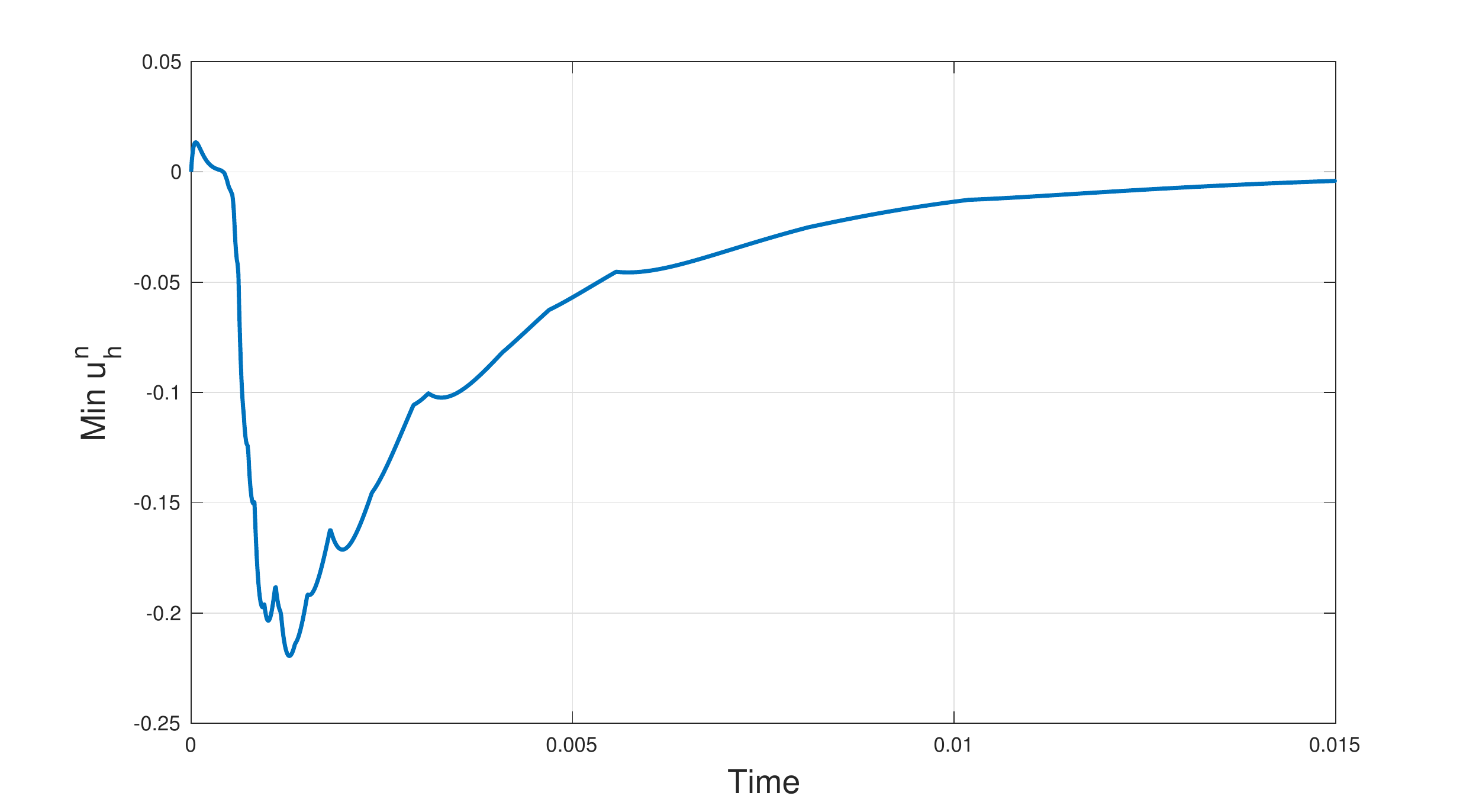}}
  \caption{Minimum values of $u^n$ computed using the scheme \textbf{BEUV}.
  \label{fig:MUbeuv}}
  \end{center}
\end{figure}

\subsection{Energy stability}
In this subsection, we compare numerically the stability of the schemes $\textbf{UV}$, $\textbf{UZSW}$, $\textbf{US}$ and $\textbf{BEUV}$ with respect to the ``exact'' energy
\begin{equation}\label{EComun}
 \mathcal{E}_e(u,v)= \displaystyle \int_\Omega  F_0(u(\x)) d\x + \frac{1}{2} \Vert \nabla {v}\Vert_0^2,
 \end{equation} 
where 
\begin{equation*}\label{F0} 
F_0(u) := F(u_+)=
\left\{\begin{array}{l}
1,  \ \ \mbox{ if } u\leq 0,\\
u ln(u) - u + 1, \ \ \mbox{ if } u>0.\\
\end{array}\right.
\end{equation*} 
We recall it was proved that the schemes $\textbf{UV}$, $\textbf{UZSW}$ and $\textbf{US}$ are unconditionally energy-stables with respect to modified energies obtained in terms of the variables of each scheme. Even more, some energy inequalities are satisfied (see Theorems \ref{estinc1uv}, \ref{estinc1us} and \ref{estinc1uzsw}). However, it is not clear how to prove the energy-stability of these schemes with respect to the ``exact'' energy $ \mathcal{E}_e(u,v)$ given in (\ref{EComun}), which comes from the continuous problem (\ref{modelf00}) (see (\ref{deluvintro})). Therefore, it is interesting to compare numerically the schemes with respect to this energy $\mathcal{E}_e(u^n,v^n)$, and to study the behaviour of the corresponding discrete residual of the energy law (\ref{deluvintro}): 
 \begin{equation}\label{ns01-bb}
RE_e(u^n,v^n):=\delta_t \mathcal{E}_e(u^n,v^n)+ 4\int_\Omega \vert \nabla \sqrt{u^{n}_+}\vert^2 d\x
+\Vert (A_h - I) v^{n}\Vert_{0}^{2}+\Vert \nabla v^{n}\Vert_{0}^{2}.
\end{equation}
1. {\it First test:} We consider $k=10^{-3}$, $h=\frac{1}{40}$, $tol=10^{-4}$ and the initial conditions (see Fig.~\ref{fig:initcond2})
$$u_0=7 w+7.0001 \ \ \mbox{ and } \ \  v_0=-7w +7.0001,$$ 
where $w:= cos(2\pi x) cos(2\pi y)$. We choose $V_h$ generated by $\mathbb{P}_2$-continuous FE. Then, we obtain that: 
\begin{figure}[htbp]
\centering 
\subfigure[Initial cell density $u_0$]{\includegraphics[width=70mm]{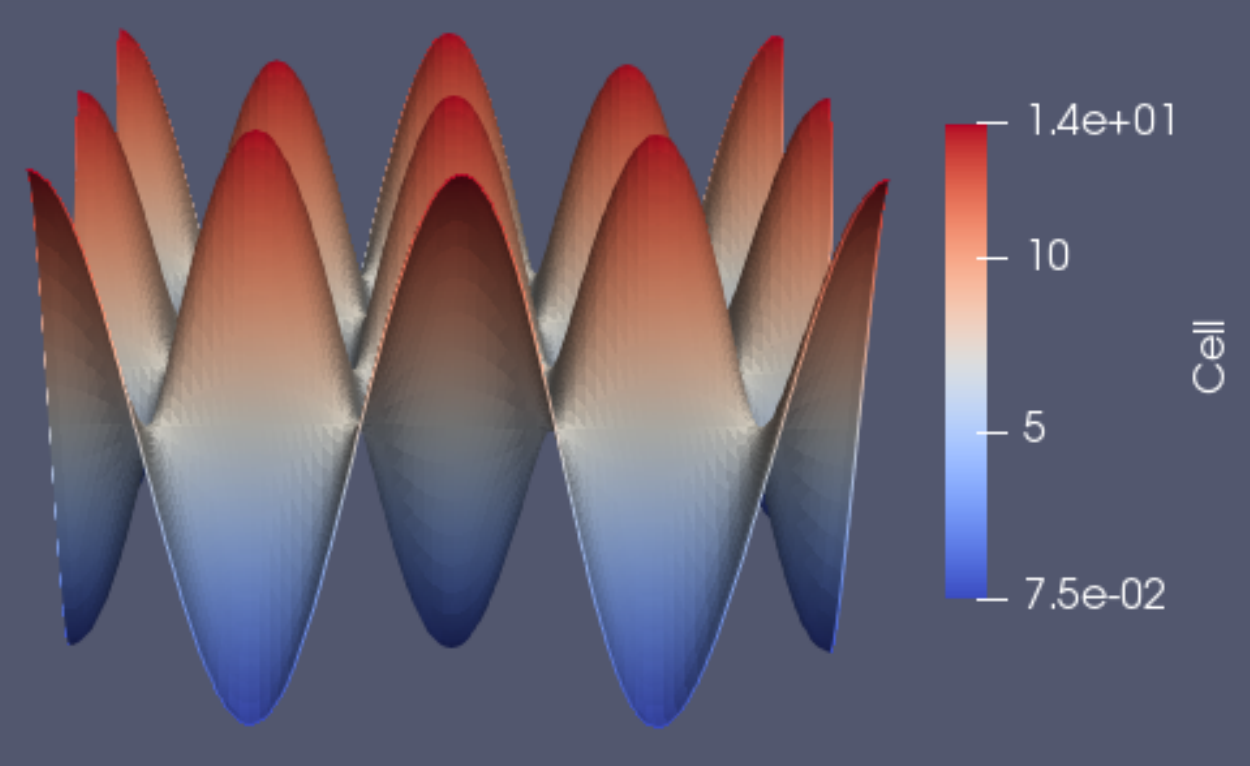}} \hspace{1,2 cm} 
\subfigure[Initial chemical concentration $v_0$]{\includegraphics[width=70mm]{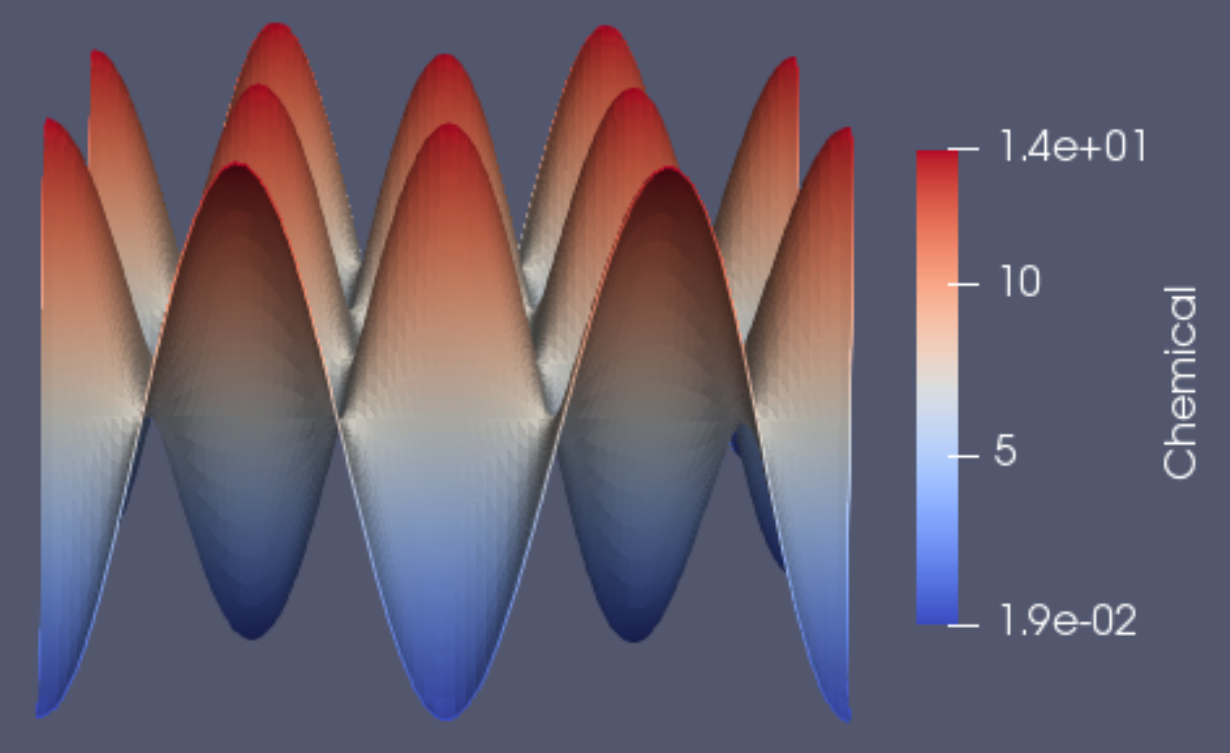}} 
\caption{Initial conditions.} \label{fig:initcond2}
\end{figure}
\begin{enumerate}
\item[(i)] The scheme $\textbf{BEUV}$ satisfies the energy decreasing in time property 
 for the exact energy $\mathcal{E}_e(u,v)$, that is, 
 \begin{equation}\label{DTPa}
 \mathcal{E}_e(u^n,v^n)\le \mathcal{E}_e(u^{n-1},v^{n-1}) \ \ \forall n.
 \end{equation}
Its behaviour can be observed in Fig.~\ref{fig:ENuvBEUV}. The same behaviour is obtained for the schemes $\textbf{UV}$ and $\textbf{US}$ independently of the choice of $\varepsilon$. In the case of the scheme \textbf{UZSW}, this property (\ref{DTPa}) is not satisfied for any value of $\varepsilon$. Indeed, increasing energies are obtained for different values of $\varepsilon$ (see Fig.~\ref{fig:ENuvUZSW}).
\item[(ii)] The scheme $\textbf{BEUV}$ satisfies the discrete energy inequality $RE_e(u^n,v^n)\leq 0$ for $RE_e(u^n,v^n)$ defined in (\ref{ns01-bb}) (see Fig. \ref{fig:REuvBEUV}). The same is observed for the schemes $\textbf{UV}$ and $\textbf{US}$ independently of the choice of $\varepsilon$. In the case of the scheme \textbf{UZSW}, it is observed that this discrete energy inequality is not satisfied for any value of $\varepsilon$. Indeed, the residual $RE_e(u^n_\varepsilon,v^n_\varepsilon)$ obtained for each $\varepsilon$ reaches very large positive values (see Fig. \ref{fig:REuvUZSW}).
\end{enumerate}
 \begin{figure}[h]
  \begin{center}
  {\includegraphics[height=0.5\linewidth]{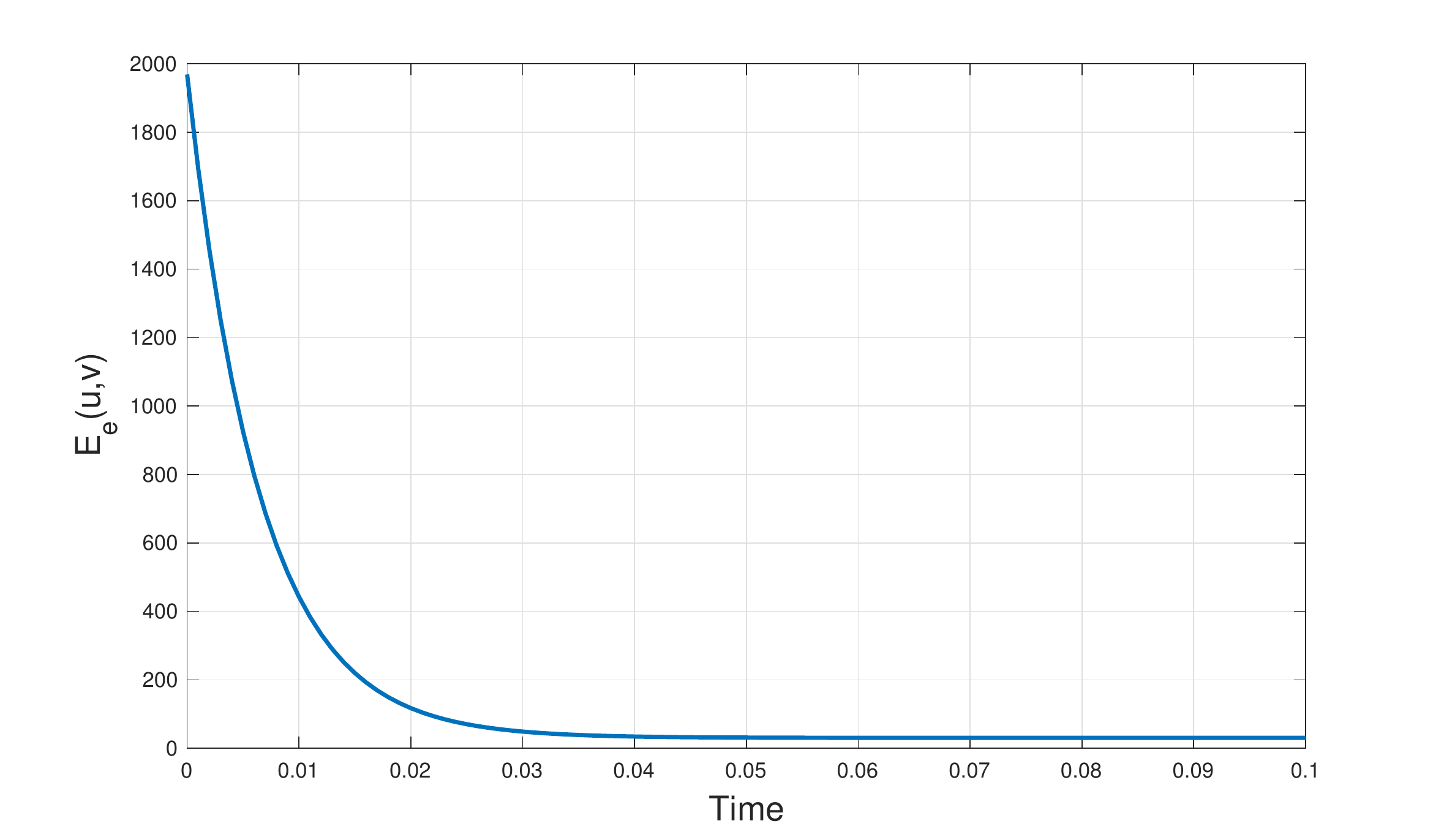}}
  \caption{Energy $\mathcal{E}_e(u^n,v^n)$ of the scheme $\textbf{BEUV}$.
  \label{fig:ENuvBEUV}}
  \end{center}
\end{figure}
 \begin{figure}[h]
  \begin{center}
  {\includegraphics[height=0.5\linewidth]{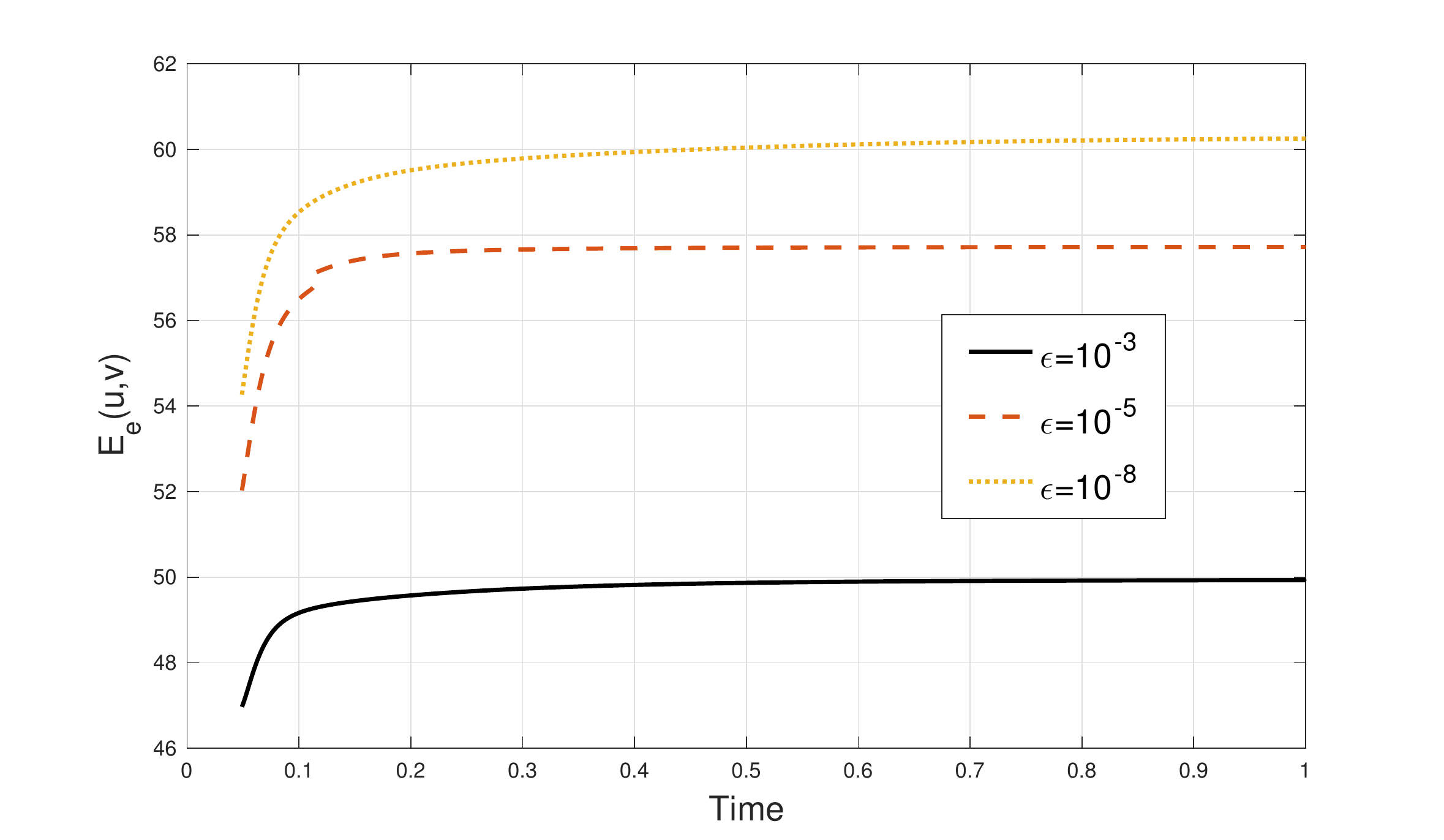}}
  \caption{Energy $\mathcal{E}_e(u^n_\varepsilon,v^n_\varepsilon)$ of the scheme $\textbf{UZSW}$ for different values of $\varepsilon$.
  \label{fig:ENuvUZSW}}
  \end{center}
\end{figure}
 \begin{figure}[h]
  \begin{center}
  {\includegraphics[height=0.5\linewidth]{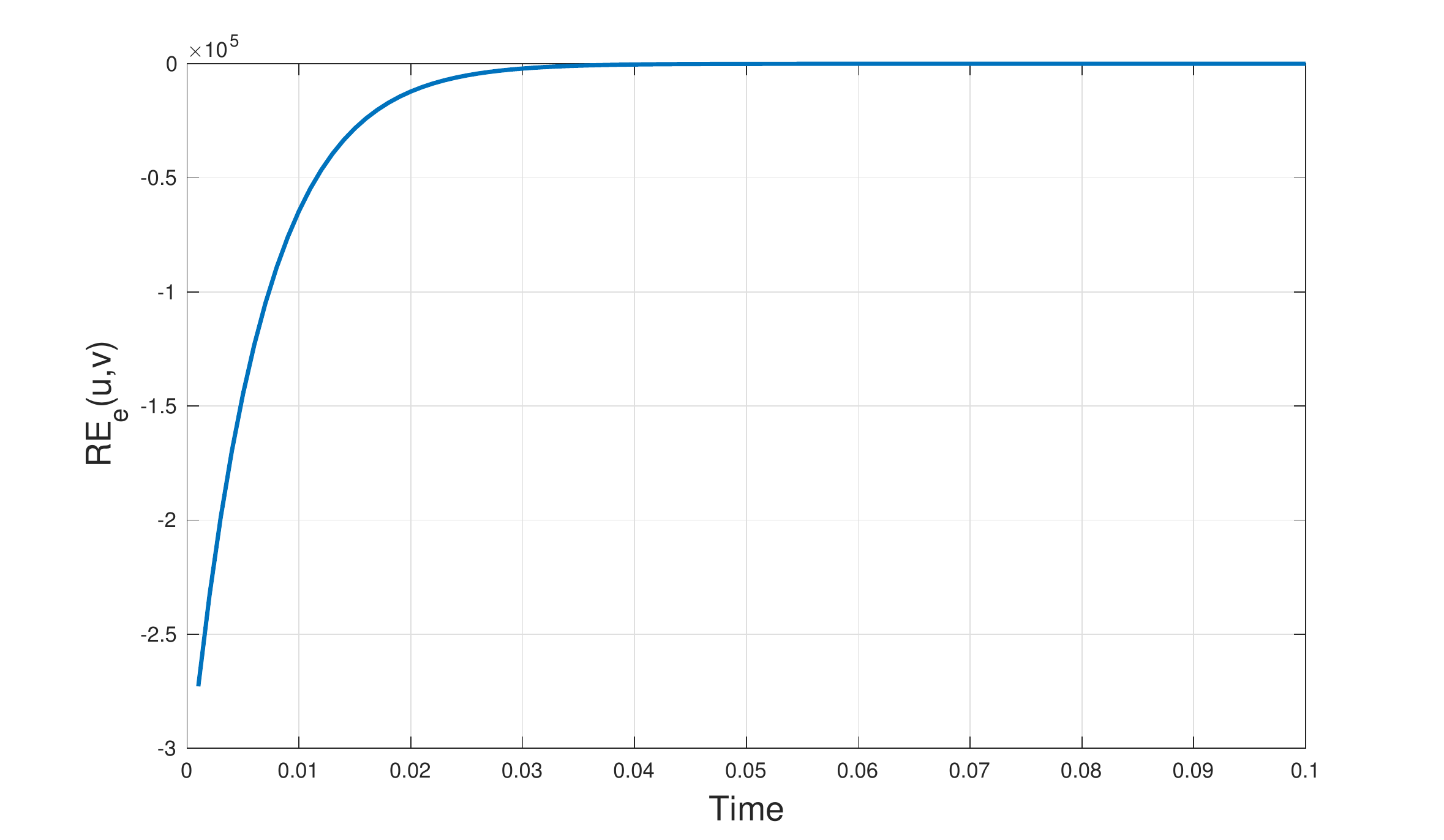}}
  \caption{$RE_e(u^n,v^n)$ of the scheme $\textbf{BEUV}$ (with approximation $\mathbb{P}_2$-continuous for $V_h$).
  \label{fig:REuvBEUV}}
  \end{center}
\end{figure}
 \begin{figure}[h]
  \begin{center}
  {\includegraphics[height=0.5\linewidth]{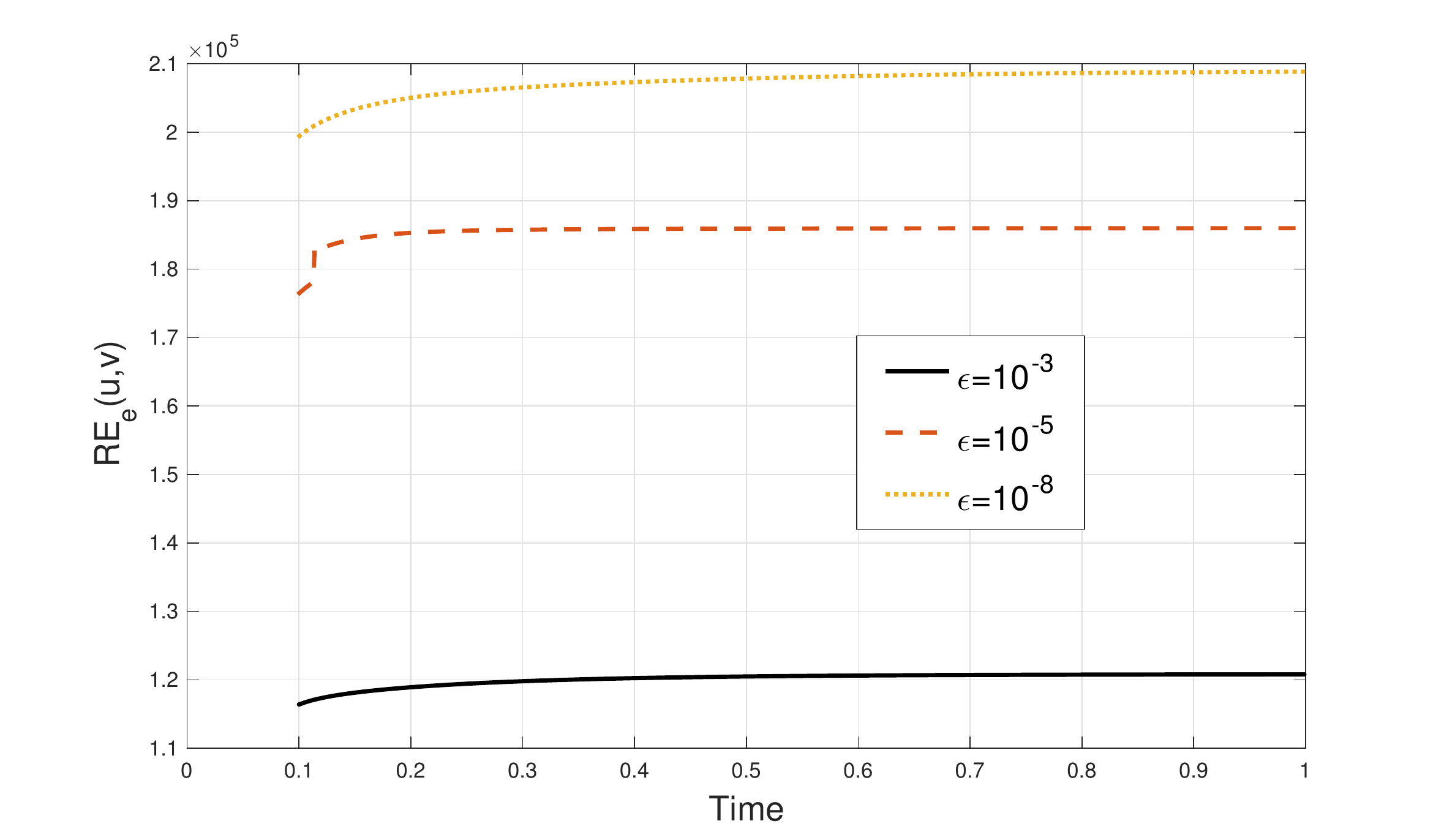}}
  \caption{$RE_e(u^n_\varepsilon,v^n_\varepsilon)$ of the scheme \textbf{UZSW} for different values of $\varepsilon$.
  \label{fig:REuvUZSW}}
  \end{center}
\end{figure}

2. {\it Second test:} We consider $k=10^{-5}$, $h=\frac{1}{20}$, $tol=10^{-4}$ and the initial conditions 
$$u_0=14w+14.0001\ \ \mbox{ and } \ \ v_0=-14w +14.0001,$$
with the function $w$ as before. Now, we choose the space $V_h$  generated by $\mathbb{P}_1$-continuous FE. Then, we obtain that: 
\begin{enumerate}
\item[(i)] The schemes $\textbf{BEUV}$, $\textbf{UV}$ and $\textbf{US}$ satisfy the energy decreasing in time property (\ref{DTPa}), independently of the choice of $\varepsilon$. 
\item[(ii)] The schemes $\textbf{UV}$ and $\textbf{US}$ satisfy the discrete energy inequality $RE_e(u^n_\varepsilon,v^n_\varepsilon)\leq 0$, independently of the choice of $\varepsilon$; while the scheme $\textbf{BEUV}$ have $RE(u^n,v^n)>0$ for some $n\geq 0$ (see Fig. \ref{fig:REuvUV2}).
\end{enumerate}
 \begin{figure}[h]
  \begin{center}
  {\includegraphics[height=0.5\linewidth]{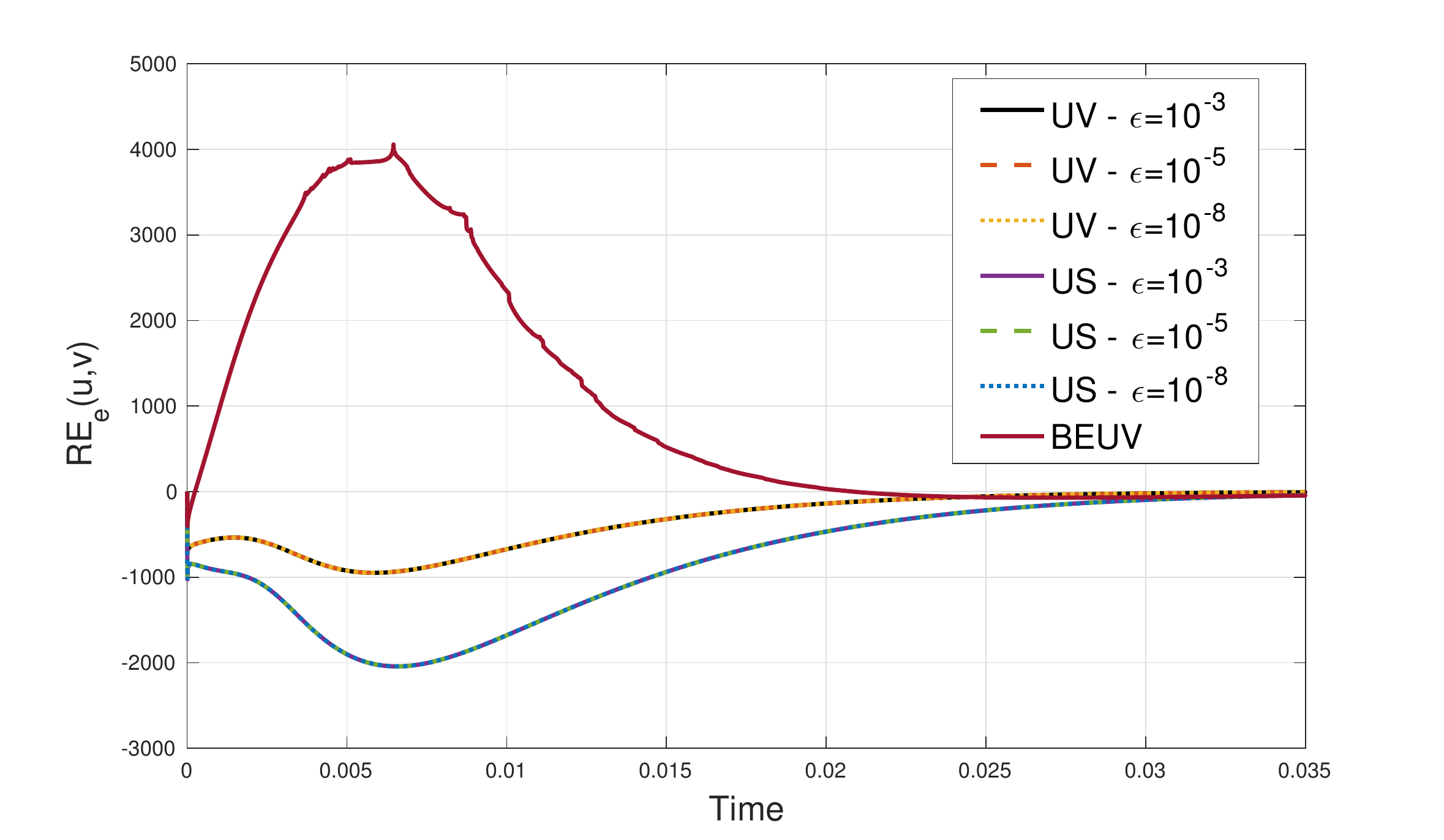}}
  \caption{$RE_e(u^n,v^n)$ of the schemes $\textbf{BEUV}$, \textbf{UV}, and \textbf{US} (with approximation $\mathbb{P}_1$-continuous for $V_h$). On the bottom, scheme \textbf{US} (for $\varepsilon=10^{-3},10^{-5},10^{-8}$); in the middle, scheme \textbf{UV} (for $\varepsilon=10^{-3},10^{-5},10^{-8}$); and on the top, scheme \textbf{BEUV}.
  \label{fig:REuvUV2}}
  \end{center}
\end{figure}

\section{Conclusions}
In this paper we have developed three new mass-conservative and unconditionally energy-stable fully discrete FE schemes for the chemorepulsion production model (\ref{modelf00}), namely \textbf{UV}, \textbf{US} and \textbf{UZSW}.
From the theoretical point of view we have obtained: 
\begin{enumerate}
\item[(i)] The well-posedness of the nume\-ri\-cal schemes (with conditional uniqueness for the nonlinear schemes \textbf{UV} and \textbf{US}).
\item[(ii)] The nonlinear scheme \textbf{UV} is unconditional energy-stable with respect to the energy $\mathcal{E}_\varepsilon^h(u,v)$ given in (\ref{Euv}), under the constraint ({\bf H}) on the space triangulation related with the right-angles and assuming that $U_h$ is approximated by $\mathbb{P}_1$-continuous FE. 
\item[(iii)] The nonlinear scheme \textbf{US} and the linear scheme \textbf{UZSW} are unconditional energy-stables with respect to the modified energies $\mathcal{E}^h_\varepsilon(u,{\boldsymbol\sigma})$ (given in (\ref{Eus})) and $\mathcal{E}(w,{\boldsymbol\sigma})$ (given in (\ref{Euzsw})) respectively, without the constraint on the triangulation related with the right-angles simplices and assuming that $U_h$ can be approximated by $\mathbb{P}_1$-continuous and $\mathbb{P}_k$-continuous FE respectively, for any $k\geq 1$. 
\item[(iv)] It is not clear how to prove the energy-stability of the nonlinear scheme \textbf{BEUV} with respect to the energy $\mathcal{E}_e(u,v)$ (given in (\ref{EComun})) or some modified energy (see Remark \ref{RBE}).
\item[(v)] In the schemes \textbf{UV} and \textbf{US} there is a control for $\Pi^h (u^n_{\varepsilon-})$ in $L^2$-norm, which tends to $0$ as $\varepsilon\rightarrow 0$. This allows to conclude the nonnegativity of the solution $u^n_\varepsilon$ in the limit when $\varepsilon\rightarrow 0$. This property is not clear for the linear scheme \textbf{UZSW}.
\end{enumerate}
On the other hand, from the numerical simulations, we can conclude:
\begin{enumerate}
\item[(i)] There are initial conditions for which the scheme \textbf{UZSW} is not energy stable with respect to the energy $\mathcal{E}_e(u,v)$, that is, the decreasing in time property (\ref{DTPa}) is not satisfied for any value of $\varepsilon$. Indeed, time increasing energies are obtained for different values of $\varepsilon$.
\item[(ii)] For the three compared nonlinear schemes (\textbf{UV}, \textbf{US} and \textbf{BEUV}), only the scheme \textbf{US} has convergence problems for the linear iterative method. However, these problems are overcomed considering thinner meshes. 
\item[(iii)] The schemes \textbf{UV} and \textbf{US} have decreasing in time energy $\mathcal{E}_e(u,v)$, independently of the choice of $\varepsilon$. In fact, the discrete energy inequality $RE_e(u^n_\varepsilon,v^n_\varepsilon)\leq 0$ is satisfied in all cases, for $RE_e(u^n_\varepsilon,v^n_\varepsilon)$ defined in (\ref{ns01-bb}).   
\item[(iv)] The scheme \textbf{BEUV} has decreasing in time energy $\mathcal{E}_e(u,v)$,  but the discrete energy inequa\-li\-ty  $RE_e(u^n,v^n)\leq 0$ is not satisfied for some $n\geq 0$.
\item[(v)] Finally, it was observed numerically that, for the schemes \textbf{UV} and \textbf{US}, $u^n_{\varepsilon-} \rightarrow 0$ as $\varepsilon\rightarrow 0$; while for the scheme \textbf{UZSW} this behavior was not observed.
\end{enumerate}

\end{document}